\documentclass{amsart}

\usepackage{amsmath}
\usepackage{amsfonts}
\usepackage{amssymb,enumerate}
\usepackage{amsthm}
\usepackage[all]{xy}
\usepackage{hyperref}

\newtheorem{lem}{Lemma}[section]
\newtheorem{cor}[lem]{Corollary}
\newtheorem{prop}[lem]{Proposition}
\newtheorem{thm}[lem]{Theorem}

\newtheorem{Defn}[lem]{Definition}
\newtheorem{Ex}[lem]{Example}
\newtheorem{Question}[lem]{Question}
\newtheorem{Questions}[lem]{Questions}
\newtheorem{Property}[lem]{Property}
\newtheorem{Properties}[lem]{Properties}
\newtheorem{Discussion}[lem]{Remark}
\newtheorem{Construction}[lem]{Construction}
\newtheorem{Notation}[lem]{Notation}
\newtheorem{Fact}[lem]{Fact}
\newtheorem{Notationdefinition}[lem]{Definition/Notation}
\newtheorem{Remarkdefinition}[lem]{Remark/Definition}
\newtheorem{Subprops}{}[lem]
\newtheorem{Para}[lem]{}

\newenvironment{defn}{\begin{Defn}\rm}{\end{Defn}}

\newenvironment{question}{\begin{Question}\rm}{\end{Question}}
\newenvironment{questions}{\begin{Questions}\rm}{\end{Questions}}

\newenvironment{properties}{\begin{Properties}\rm}{\end{Properties}}

\newenvironment{notationdefinition}{\begin{Notationdefinition}\rm}{\end{Notationdefinition}}

\newenvironment{subprops}{\begin{Subprops}\rm}{\end{Subprops}}
\newenvironment{para}{\begin{Para}\rm}{\end{Para}}
\newenvironment{disc}{\begin{Discussion}\rm}{\end{Discussion}}

\newtheorem{intthm}{Theorem}

\newcommand{\D}{\mathcal{D}}

\newcommand{\cat}[1]{\mathcal{#1}}

\newcommand{\catp}{\cat{P}}
\newcommand{\catf}{\cat{F}}
\newcommand{\cati}{\cat{I}}
\newcommand{\cata}{\cat{A}}

\newcommand{\catb}{\cat{B}}

\newcommand{\catac}{\cat{A}_C}
\newcommand{\catab}{\cat{A}_B}
\newcommand{\catbc}{\cat{B}_C}

\newcommand{\catcif}{\operatorname{CI}\text{-}\catf}
\newcommand{\catucii}{\operatorname{CI}^*\!\!\text{-}\cati}

\newcommand{\pd}{\operatorname{pd}}	
\newcommand{\gdim}{\mathrm{G}\text{-}\!\dim}	
\newcommand{\gkdim}[1]{\mathrm{G}_{#1}\text{-}\!\dim}	
\newcommand{\gcdim}{\gkdim{C}}	
	
\newcommand{\id}{\operatorname{id}}	
\newcommand{\fd}{\operatorname{fd}}

\newcommand{\gpd}{\text{G-}\!\pd}
\newcommand{\gfd}{\text{G-}\!\fd}
\newcommand{\gcpd}{\text{G}_C\text{-}\!\pd}
\newcommand{\gcfd}{\text{G}_C\text{-}\!\fd}
\newcommand{\gcid}{\text{G}_C\text{-}\!\id}
\newcommand{\gkpd}[1]{\text{G}_{#1}\text{-}\!\pd}
\newcommand{\gkfd}[1]{\text{G}_{#1}\text{-}\!\fd}

\newcommand{\gid}{\text{G-}\!\id}

\newcommand{\cidim}{\mathrm{CI}\text{-}\!\dim}	
\newcommand{\cipd}{\mathrm{CI}\text{-}\!\pd}	
\newcommand{\cifd}{\mathrm{CI}\text{-}\!\fd}	
\newcommand{\ciid}{\mathrm{CI}\text{-}\!\id}	
\newcommand{\uciid}{\mathrm{CI}^*\!\text{-}\!\id}

\newcommand{\depth}{\operatorname{depth}}	
	
\newcommand{\amp}{\operatorname{amp}}
\newcommand{\edim}{\operatorname{edim}}

\newcommand{\ext}{\operatorname{Ext}}	
\newcommand{\rhom}{\mathbf{R}\!\operatorname{Hom}}	
\newcommand{\lotimes}{\otimes^{\mathbf{L}}}
\newcommand{\HH}{\operatorname{H}}
\newcommand{\Hom}{\operatorname{Hom}}	
\newcommand{\coker}{\operatorname{Coker}}
\newcommand{\spec}{\operatorname{Spec}}

\newcommand{\tor}{\operatorname{Tor}}

\newcommand{\shift}{\mathsf{\Sigma}}

\newcommand{\Ker}{\operatorname{Ker}}

\newcommand{\ideal}[1]{\mathfrak{#1}}
\newcommand{\m}{\ideal{m}}
\newcommand{\n}{\ideal{n}}
\newcommand{\p}{\ideal{p}}
\newcommand{\q}{\ideal{q}}

\newcommand{\comp}[1]{\widehat{#1}}

\newcommand{\Min}{\operatorname{Min}}

\newcommand{\from}{\leftarrow}
\newcommand{\xra}{\xrightarrow}
\newcommand{\xla}{\xleftarrow}

\newcommand{\vf}{\varphi}

\newcommand{\x}{\mathbf{x}}

\renewcommand{\geq}{\geqslant}
\renewcommand{\leq}{\leqslant}

\begin{document}

\bibliographystyle{amsplain}

\author{Sean Sather-Wagstaff}

\address{Sean Sather-Wagstaff, Department of Mathematics,
300 Minard Hall,
North Dakota State University,
Fargo, North Dakota 58105-5075, USA}

\email{Sean.Sather-Wagstaff@ndsu.edu}

\urladdr{http://math.ndsu.nodak.edu/faculty/ssatherw/}

\title{Complete intersection dimensions and Foxby classes}

\date{\today}

\dedicatory{Dedicated to Luchezar L.\ Avramov on the occasion of his sixtieth birthday}

\keywords{Auslander classes, Bass classes, complete intersection dimensions,
Foxby classes, Foxby equivalence, quasi-deformations, semidualizing complexes}
\subjclass[2000]{13A35, 13B10, 13C05, 13D05, 13D07, 13D25, 14B25}

\begin{abstract}
Let $R$ be a local ring and $M$ a finitely generated $R$-module.
The complete intersection dimension of $M$---defined by Avramov, Gasharov and Peeva,
and denoted $\cidim_R(M)$---is a homological invariant
whose finiteness implies that $M$ is similar to 
a module over a complete intersection. 
It is related to the classical projective dimension
and to Auslander and Bridger's Gorenstein dimension 
by the inequalities
$\gdim_R(N)\leq\cidim_R(N)\leq\pd_R(N)$.

Using Blanco and Majadas' version of complete intersection dimension
for local ring homomorphisms, 
we prove the following generalization of a theorem
of Avramov and Foxby:  
Given local ring homomorphisms
$\varphi\colon R\to S$ and 
$\psi\colon S\to T$ such that $\varphi$
has finite Gorenstein dimension, if $\psi$ has finite complete intersection
dimension, then the composition $\psi\circ\varphi$ has finite Gorenstein dimension.
This follows from our result stating that, if $M$ has finite complete
intersection dimension, then
$M$ is $C$-reflexive and is in the Auslander class $\catac(R)$
for each semidualizing $R$-complex $C$.  
\end{abstract}

\maketitle

\section*{Introduction} \label{sec0}

Let $R$ be a local ring and $N$ a finitely generated $R$-module. 
The projective dimension of $N$, denoted $\pd_R(N)$, is by now a classical invariant,
and much research has shown that modules of 
finite projective dimension have properties similar to 
those of modules over a regular local ring.
Motivated by this,
Auslander and Bridger~\cite{auslander:smt}
introduced the Gorenstein dimension of $N$,
denoted $\gdim_R(N)$, which is 
an invariant whose finiteness detects properties similar to those for
modules over a Gorenstein ring. 
More recently Avramov, Gasharov and Peeva~\cite{avramov:cid}
defined the complete intersection dimension of $N$,
denoted $\cidim_R(N)$, which plays a similar role with respect to
the complete intersection property.
Corresponding to the well-known hierarchy of rings, there are
inequalities
$$\gdim_R(N)\leq\cidim_R(N)\leq\pd_R(N)$$
with equality to the left of any finite quantity.
See Sections~\ref{sec1} and~\ref{sec2} for foundations of
Gorenstein dimensions and complete intersection
dimensions.

Avramov and Foxby~\cite{avramov:rhafgd} used Auslander and Bridger's
Gorenstein dimension to define what it means for  a local ring homomorphism
$\vf\colon R\to S$ to have finite Gorenstein dimension.  
Note that one cannot simply define the Gorenstein dimension of $\vf$ to be $\gdim_R(S)$,
as $S$ is not assumed to be finitely generated as an $R$-module.
Avramov and Foxby overcome this technical difficulty by using the 
Cohen factorizations
of Avramov, Foxby and Herzog~\cite{avramov:solh}
to replace $\vf$ with a related homomorphism
$\vf'\colon R'\to\comp S$ which has the added benefit of being surjective
so that $\gdim_{R'}(\comp S)$ is defined.
Blanco and Majadas~\cite{blanco:miccp} use the same idea to
study local homomorphisms of  finite complete intersection.

Avramov and Foxby~\cite{avramov:rhafgd} asked the following:
given two local ring homomorphisms $\vf\colon R\to S$ and $\psi\colon S\to T$
of finite Gorenstein dimension, must the composition $\psi \circ\vf$
also have finite Gorenstein dimension?
They were able to answer this question, for example, when
$\psi$ has finite flat dimension.  
We move one step closer to answering this question 
in general with the following
result; see Theorem~\ref{compose01}.

\begin{intthm} \label{thmB}
Let $\vf\colon R\to S$ and $\psi\colon S\to T$ be local ring homomorphisms.
If $\vf$ has finite Gorenstein dimension and $\psi$ has finite complete intersection
dimension, then the composition $\psi \circ\vf$ has finite Gorenstein dimension.
\end{intthm}

We also establish the following complete-intersection analogues
of results of Avramov, Iyengar and Miller~\cite{avramov:holh}
and Foxby and Frankild~\cite{foxby:cmfgidgr}.
The first is proved in~\ref{ciid01}. 
The second one is contained in Corollary~\ref{ciid03} and can also be thought of as an
injective version of a result of Blanco and Majadas~\cite{blanco:miccp}.

\begin{intthm} \label{thmD}
Let $\vf\colon R\to S$ be a local ring homomorphism.
Then $\vf$ has finite complete intersection injective
dimension
if and only if $R$ is Gorenstein
and $\vf$ has finite complete intersection
dimension.
\end{intthm}

\begin{intthm} \label{thmE}
A local ring  of prime characteristic is a complete
intersection if and only if some (equivalently, every) power of its Frobenius 
endomorphism has finite complete intersection injective dimension.
\end{intthm}

Theorem~\ref{thmB} is proved using the 
Auslander classes introduced 
by Avramov and Foxby~\cite{avramov:rhafgd} and generalized
by Christensen~\cite{christensen:scatac}; see~\ref{sdc05}.  
Much recent research has been devoted to the study of these classes,
not only because of their connection to the composition question of Avramov and Foxby,
but also because the objects in these classes enjoy particularly nice
homological properties.  

Each Auslander class
of $R$-complexes contains every
bounded $R$-module of finite flat dimension.  The next result 
greatly enlarges the class of objects known to be in
each Auslander class; it is contained in Theorem~\ref{cifdac01}.

\begin{intthm} \label{thmA}
If $R$ is a local ring and $M$ an $R$-module of 
finite complete intersection flat dimension, then $M$ is in the Auslander class
$\catac(R)$ for each semidualizing $R$-complex $C$.
\end{intthm}

Here, the complete intersection flat dimension of $M$ is a 
version of complete intersection dimension for modules that are not 
necessarily finitely generated.
We actually prove 
a more general result for $R$-complexes
and also a dual result  in terms of upper complete intersection injective dimension 
$\uciid_R(M)$ and
the Bass classes $\catbc(R)$.  In particular, when $R$ admits a dualizing complex $D^R$,
these complete intersection dimensions determine natural subcategories
$\catcif(R)\subseteq \cata_{D^R}(R)$ and $\catucii(R)\subseteq\catb_{D^R}(R)$.
The next result shows that ``Foxby equivalence'' between the categories
$\cata_{D^R}(R)$ and  $\catb_{D^R}(R)$ 
restricts to an equivalence of categories
$\catcif(R)\sim \catucii(R)$; see~\ref{foxby01} for the proof.

\begin{intthm} \label{thmC}
Let $R$ be a local ring admitting a dualizing complex $D^R$,
and let $M$ be a homologically bounded $R$-complex.
\begin{enumerate}[\quad\rm(a)]
\item \label{thmCitem1}
$\cifd_R(M)<\infty$ if and only if $\uciid_R(D^R\lotimes_RM)<\infty$.
\item \label{thmCitem2}
$\uciid_R(M)<\infty$ if and only if $\cifd_R(\rhom_R(D^R,M))<\infty$.
\end{enumerate}
\end{intthm}

A  special case of this theorem
augments a result of Levin and Vasconcelos
and is 
in Corollary~\ref{foxby04}:  When $R$ is Gorenstein, one has
$\cifd_R(M)<\infty$ if and only if $\uciid_R(M)<\infty$.
Other results of this type are proved in Section~\ref{sec6}.

By definition, if $M$ is an $R$-complex of finite complete intersection flat dimension,
then there exists a ``quasi-deformation'' $R\xra{\vf}R'\xla{\tau}Q$ such that
$\fd_Q(R'\lotimes_R M)$ is finite;  see~\ref{cidim01} and~\ref{cidim03}.
It is commonly known that one can exert a small amount of control on
the structure of this quasi-deformation.  For instance, one may assume
without loss of generality that the closed fibre
of $\vf$ is artinian (hence, Cohen-Macaulay) and that $Q$ is complete.
One piece of technology that allows for more flexibility in our analysis
is the following result, proved in~\ref{qd02}.

\begin{intthm} \label{thmF}
Let $(R,\m)$ be a local ring and let $M$ be a homologically bounded $R$-complex.
Then $\cifd_R(M)<\infty$ if and only if there exists a quasi-deformation
$R\to R''\from Q'$ such that $Q'$ is complete, the closed fibre $R''/\m R''$ is artinian
and Gorenstein, and $\fd_{Q'}(R''\lotimes_RM)$ is finite.
\end{intthm}

Many of the results in this paper can be stated strictly in terms of modules
without losing their flavor.  However, some of our proofs \emph{require}
the use of complexes.  For this reason, we work almost entirely
in the derived-category setting.  Section~\ref{sec1} contains
a summary of the basic notions we use.

\section{Complexes and Ring Homomorphisms} \label{sec1}

Let $(R,\m,k)$ and $(S,\n,l)$ be commutative local noetherian rings.

\begin{notationdefinition} \label{basics02}
We work in the derived category $\D(R)$ of 
complexes of $R$-modules, indexed homologically.
References on the subject include~\cite{gelfand:moha, hartshorne:rad}.

A complex $M$ is \emph{homologically bounded} if $\HH_i(M)=0$ for all $|i|\gg 0$; 
and it is \emph{homologically finite} if $\oplus_i\HH_i(M)$  is finitely generated.
Let $\D_{\mathrm{b}}(R)$ denote
the full subcategory of $\D(R)$ consisting of the homologically bounded $R$-complexes.
Isomorphisms in $\D(R)$ are identified by the symbol $\simeq$,
and isomorphisms up to shift are designated by $\sim$. 

Fix $R$-complexes $M$ and $N$. Let
$\inf(M)$ and $\sup(M)$ denote the infimum and supremum, respectively, of the set
$\{n\in\mathbf{Z}\mid\HH_n(M)\neq 0\}$, and set $\amp(M)=\sup(M)-\inf(M)$.
Let $M\lotimes_R N$ and $\rhom_R(M,N)$ denote the left-derived
tensor product and right-derived homomorphism complexes, respectively.  
For each integer $i$,
the $i$th \emph{suspension} (or \emph{shift}) of
$M$, denoted $\shift^i M$, is the complex with
$(\shift^i M)_n=M_{n-i}$ and $\partial_n^{\shift^i M}=(-1)^i\partial_{n-i}^M$.
When $M$ is homologically bounded,
let $\pd_R(M)$, $\fd_R(M)$ and $\id_R(M)$ denote the projective, flat and injective
dimensions of $M$, respectively, as in~\cite{avramov:hdouc}.
Let $\catp(R)$, $\catf(R)$ and $\cati(R)$ denote the full subcategories of $\D_{\mathrm{b}}(R)$
consisting of the complexes with, respectively, finite projective, flat and injective dimension.
\end{notationdefinition}

We shall have several occasions to use the following isomorphisms from~\cite[(4.4)]{avramov:hdouc}.

\begin{notationdefinition} \label{basics03}
Let $R\to S$ be a ring homomorphism, and fix an $R$-complex $L$
and $S$-complexes $M$ and $N$.  Assume that each $R$-module
$\HH_i(L)$ is finitely generated and $\inf(L)>-\infty$.

The natural
\emph{tensor-evaluation morphism}
$$\omega_{LMN}\colon\rhom_R(L,M)\lotimes_SN\to\rhom_R(L,M\lotimes_SN)$$
is an isomorphism when $\sup(M)<\infty$ and either 
$L\in\catp(R)$ or $N\in\catf(S)$.

The natural
\emph{Hom-evaluation morphism}
$$\theta_{LMN}\colon L\lotimes_R\rhom_S(M,N)\to\rhom_{S}(\rhom_R(L,M),N)$$
is an isomorphism when $M\in\D_{\mathrm{b}}(S)$ and either 
$L\in\catp(R)$ or $N\in\cati(S)$.
\end{notationdefinition}

\begin{disc} \label{fpd01}
Let
$M$ be a homologically finite $R$-complex with $\pd_R(M)<\infty$
and let $N$ be a homologically bounded $R$-complex.
Because $M$ is homologically finite and $\pd_R(M)<\infty$, we know
from~\cite[(2.13)]{christensen:scatac} that
the $R$-complex $\rhom_R(M,R)$ is homologically finite and has finite projective dimension.
Hence, tensor-evaluation~\eqref{basics03} yields the first isomorphism in the next sequence,
and the second isomorphism
is tensor-cancellation:
$$\rhom_R(M,R)\lotimes_RN \simeq\rhom_R(M,R\lotimes_RN) \simeq\rhom_R(M,N).
$$
\end{disc}

\begin{notationdefinition} \label{basics01}
Let $\vf\colon R\to S$ be a local ring homomorphism.
We denote by $\comp R$ the completion of $R$ at its maximal ideal
and let
$\varepsilon^{}_R\colon R\to\comp{R}$ denote
the natural map.
The \emph{completion} of $\vf$ is the unique local ring homomorphism
$\comp{\vf}\colon\comp{R}\to\comp{S}$ 
such that $\comp\vf\circ\varepsilon_R=\varepsilon_S\circ\vf$.
The \emph{semi-completion} of $\vf$ is the composition
$\grave{\vf}=\varepsilon^{}_S\circ\vf\colon R\to \comp{S}$, 
and the flat dimension of $\vf$ is $\fd(\vf)=\fd_R(S)$.
By~\cite[(1.1)]{avramov:solh} the map $\grave{\vf}$ admits a 
\emph{Cohen factorization}, that is, there is a diagram of local ring homomorphisms,
$R\xrightarrow{\Dot{\vf}}R'\xrightarrow{\vf'}\comp S$, where $\grave\vf=\vf'\circ\Dot{\vf}$, with
$\Dot{\vf}$ flat, the closed fibre $R'/\m R'$ regular, $R'$ complete, and
$\vf'$ surjective.
\end{notationdefinition}

\begin{notationdefinition} \label{sdc01}
A homologically finite $R$-complex $C$ is \emph{semidualizing} if
the homothety morphism $\chi^R_C\colon R\to\rhom_R(C,C)$ is an
isomorphism in $\D(R)$.   
A complex $D$ is \emph{dualizing}
if it is semidualizing and $\id_R(D)$ is finite.
\end{notationdefinition}

\begin{disc} \label{sdc02}
Let $\vf\colon R\to S$ be a local ring homomorphism
of finite flat dimension and let $M$ be a  homologically finite $R$-complex.
From~\cite[(5.7)]{christensen:scatac} and~\cite[(4.5)]{frankild:rrhffd}
we know that $S\lotimes_R M$ is semidualizing for $S$
if and only if $M$ is semidualizing for $R$,
and $S\lotimes_R M$ is dualizing for $S$
if and only if $M$ is dualizing for $R$ and $\vf$ is Gorenstein
by~\cite[(5.1)]{avramov:lgh}.
For example, the map $\vf$ is Gorenstein 
if it is flat  with Gorenstein closed fibre~\cite[(4.2)]{avramov:lgh}
or surjective with $\Ker(\vf)$ generated by an $R$-regular sequence~\cite[(4.3)]{avramov:lgh}.
Consult~\cite{avramov:lgh} for more information on Gorenstein homomorphisms.
\end{disc}

\begin{disc} \label{dc01}
If $R$ is a homomorphic image of a Gorenstein ring, e.g., if 
$R$ is complete, then $R$ admits a dualizing complex by~\cite[(V.10.4)]{hartshorne:rad}.
\end{disc}

\begin{notationdefinition} \label{sdc03}
Let $C$ be a semidualizing $R$-complex.  A homologically finite $R$-complex
$M$ is \emph{$C$-reflexive} if the complex $\rhom_R(M,C)$ is homologically
bounded and the biduality morphism 
$\delta^C_M\colon M\to\rhom_R(\rhom_R(M,C),C)$
is an isomorphism in $\D(R)$.  Set
$$\gcdim_R(M):=\begin{cases}
\inf(C)-\inf(\rhom_R(M,C)) & \text{if $M$ is $C$-reflexive} \\
\infty & \text{otherwise.}
\end{cases}$$
When $C=R$ we write $\gdim_R(M)$ in lieu of $\gkdim{R}_R(M)$;
this is the \emph{G-dimension} of
Auslander and Bridger~\cite{auslander:smt}
and Yassemi~\cite{yassemi:gd}.
\end{notationdefinition}

\begin{disc} \label{sdc04}
Assume that $R$ admits a dualizing
complex $D$.  Each homologically finite $R$-complex $M$ is
$D$-reflexive by~\cite[(V.2.1)]{hartshorne:rad}, and~\cite[(2.12)]{christensen:scatac} 
tells us that $M$ is
semidualizing for $R$ if and only if $\rhom_R(M,D)$ is so.
\end{disc}

\begin{defn} \label{sdm01}
Let $C$ be a semidualizing $R$-module.  

An $R$-module $G$ is \emph{$\text{G}_C$-projective}
if there exists an exact sequence of $R$-modules
$$X=\cdots\xra{\partial^X_2}P_1
\xra{\partial^X_1} P_0 
\xra{\partial^X_0} C\otimes_R P_{-1}
\xra{\partial^X_{-1}} C\otimes_R P_{-2}
\xra{\partial^X_{-2}}\cdots$$
such that 
$G\cong\coker(\partial^X_1)$, 
each $P_i$ is projective, 
and  $\Hom_R(X,C\otimes_R Q)$ is exact for each 
projective $R$-module $Q$.

An $R$-module $G$ is \emph{$\text{G}_C$-flat}
if there exists an exact sequence  of $R$-modules
$$Y=\cdots\xra{\partial^X_2}F_1
\xra{\partial^Y_1} F_0 
\xra{\partial^Y_0} C\otimes_R F_{-1}
\xra{\partial^Y_{-1}} C\otimes_R F_{-2}
\xra{\partial^Y_{-2}}\cdots$$
such that $G\cong\coker(\partial^Y_1)$, each $F_i$ is flat,  and 
$\Hom_R(C,I)\otimes_R Y$ is exact for each 
injective $R$-module $I$.

An $R$-module $G$ is \emph{$\text{G}_C$-injective}
if there exists an exact sequence  of $R$-modules
$$Z=\cdots\xra{\partial^Z_2}\Hom_R(C,I_1)
\xra{\partial^Z_1} \Hom_R(C,I_0)
\xra{\partial^Z_0} I_{-1}
\xra{\partial^Z_{-1}} I_{-2}
\xra{\partial^Z_{-2}}\cdots$$
such that $G\cong\coker(\partial^Z_1)$, each $I_i$ is  injective,  and 
$\Hom_R(\Hom_R(C,I),Z)$ is exact for each 
injective $R$-module $I$.

Let $M$ be a homologically bounded
$R$-complex.  A \emph{$\text{G}_C$-projective resolution} of $M$ is an isomorphism
$H\simeq M$ in $\D(R)$ where $H$ is a complex of $\text{G}_C$-projective $R$-modules
such that $H_i=0$ for all $i\ll 0$.  The \emph{$\text{G}_C$-projective dimension}
of $M$ is 
$$\gcpd_R(M):=
\inf\{\sup\{n\mid H_n\neq 0\}\mid \text{$H\simeq M$ is a $\text{G}_C$-projective resolution}\}.
$$
The \emph{$\text{G}_C$-flat dimension} of $M$ is defined similarly and denoted
$\gcfd_R(M)$, while the \emph{$\text{G}_C$-injective dimension} $\gcid_R(M)$
is dual.
When $C=R$ we write $\gpd_R(M)$, $\gfd_R(M)$ and $\gid_R(M)$;
these are the \emph{G-projective, G-flat, and G-injective dimensions} 
of Enochs, Jenda and Torrecillas~\cite{enochs:gipm,enochs:gf}.
\end{defn}

\begin{disc} \label{sdm02}
Let $C$ be a semidualizing $R$-module, and let $R\ltimes C$ denote the
trivial extension of $R$ by $C$.  Let $M$ be a
homologically bounded $R$-complex, and view $M$ as an $R\ltimes C$-complex
via the natural surjection $R\ltimes C\to R$.
From~\cite[(2.16)]{holm:smarghd}
there are equalities
\begin{gather*}
\gcpd_R(M)=\gpd_{R\ltimes C}(M) \qquad\qquad
\gcfd_R(M)=\gfd_{R\ltimes C}(M) \\
\gcid_R(M)=\gid_{R\ltimes C}(M).
\end{gather*}
It is known in a number of cases that the quantities
$\gcpd_R(M)$ and $\gcfd_R(M)$ are simultaneously finite.
When $C=R$ and $R$ admits a dualizing complex, this is in~\cite[(4.3)]{christensen:ogpifd}.
When $M$ is a module, it is in~\cite[(3.5)]{esmkhani:ghdac} and~\cite[(3.3)]{sather:crct}.
We deal with the general case in Proposition~\ref{sdm03}.
\end{disc}

The next categories 
come from~\cite{avramov:rhafgd,christensen:scatac}
and are commonly known as Foxby classes. 

\begin{notationdefinition} \label{sdc05}
Let $C$ be a semidualizing $R$-complex.  

The
\emph{Auslander class} with respect to $C$ is the full subcategory 
$\catac(R)\subseteq\D_{\mathrm{b}}(R)$ consisting of the complexes $M$ such that
$C\lotimes_R M\in\D_{\mathrm{b}}(R)$
and the natural morphism $\gamma^C_M\colon M\to\rhom_R(C,C\lotimes_R M)$
is an isomorphism in $\D(R)$.

The
\emph{Bass class} with respect to $C$ is the full subcategory 
$\catbc(R)\subseteq\D_{\mathrm{b}}(R)$ consisting of the complexes $N$ such that
$\rhom_R(C, N)\in\D_{\mathrm{b}}(R)$
and the natural morphism $\xi^C_N\colon C\lotimes_R \rhom_R(C,M)\to M$
is an isomorphism in $\D(R)$.
\end{notationdefinition}

\begin{disc} \label{lwc01}
Let $C$ be a semidualizing $R$-complex, and let $X$ be a homologically bounded
$R$-complex.  
If $\fd_R(X)<\infty$, then $X\in\catac(R)$; if $\id_R(X)<\infty$, then $X\in\catbc(R)$; 
see~\cite[(4.4)]{christensen:scatac}.

Let $\vf\colon R\to S$ be a local homomorphism such that $S\in\catac(R)$, e.g.,
such that $\fd(\vf)<\infty$.  From~\cite[(5.3)]{christensen:scatac} we learn that
$S\lotimes_RC$ is a semidualizing  $S$-complex.  Furthermore, for any $S$-complex $Y$,
we have $Y\in\catac(R)$ if and only if $Y\in\cata_{S\lotimes_RC}(S)$,
and $Y\in\catbc(R)$ if and only if $Y\in\catb_{S\lotimes_RC}(S)$.
When $\vf$ has finite flat dimension,
one has 
$X\in \catac(R)$ if and only if $S\lotimes_R X\in\cata_{S\lotimes_RC}(S)$,
and $X\in \catbc(R)$ if and only if $S\lotimes_R X\in\catb_{S\lotimes_RC}(S)$
by~\cite[(5.8)]{christensen:scatac}.
\end{disc}

\begin{disc} \label{sdm02'}
Let $C$ be a semidualizing $R$-module and assume that $R$ admits a dualizing complex $D$.
For each homologically bounded $R$-complex $M$, we have
\begin{enumerate}[\quad\rm(a)]
\item \label{sdm02'item1}
$\gcfd_R(M)<\infty$ if and only if $M\in\cata_{\rhom_R(C,D)}(R)$, and
\item \label{sdm02'item2}
$\gcid_R(M)<\infty$ if and only if $M\in\catb_{\rhom_R(C,D)}(R)$.
\end{enumerate}
This is from~\cite[(4.6)]{holm:smarghd}.
We improve upon this in Proposition~\ref{sdm03} below.
\end{disc}

\begin{disc} \label{sdm02''}
Let $C$ be a semidualizing $R$-complex and assume that $R$ admits a dualizing complex $D$.
For each homologically finite $R$-complex $M$, we have
$\gcdim_R(M)<\infty$ if and only if $M\in\cata_{\rhom_R(C,D)}(R)$
by~\cite[(4.7)]{christensen:scatac}.
\end{disc}

\section{Gorenstein and Complete Intersection Dimensions} \label{sec2}

In this section, we introduce natural variations of existing homological dimensions,
beginning with the $\text{G}_C$-version of the main player of~\cite{iyengar:golh}.

\begin{defn}  \label{gcdim01}
Let $\vf\colon R\to S$ be a local ring homomorphism and $M$ a homologically finite $S$-complex. 
Fix a semidualizing $R$-complex $C$ and a Cohen
factorization
$R\xrightarrow{\Dot{\vf}} R'\xrightarrow{\vf'} \comp{S}$ of $\grave{\vf}$.  
The \emph{$\text{G}_C$-dimension of $M$ over $\vf$} is the
quantity
$$\gcdim_{\vf}(M):= 
\gkdim{R'\lotimes_R C}_{R'}(\comp{S}\lotimes_S M)-\edim(\Dot{\vf}).$$
The \emph{$\text{G}_C$--dimension of $\vf$} is 
$\gcdim(\vf):=\gcdim_{\vf}(S)$. In the case $C=R$, we
follow~\cite{iyengar:golh} and write
$\gdim_{\vf}(M):=\gkdim{R}_{\vf}(M)$ and $\gdim(\vf):=\gkdim{R}(\vf)$.
\end{defn}

\begin{properties} \label{sri0}
Fix a local ring homomorphism $\vf\colon R\to S$, a Cohen factorization $R\to R'\to \comp{S}$
of $\grave{\vf}$, a homologically finite $S$-complex $M$, and a semidualizing
$R$-complex $C$.

\begin{subprops}  \label{sri02}
If $R\to R''\to \comp{S}$ is another Cohen
factorization of $\grave{\vf}$, then the quantities
$\gkdim{R'\lotimes_R C}_{R'}(\comp{S}\lotimes_S M)$ and
$\gkdim{R''\lotimes_R C}_{R''}(\comp{S}\lotimes_S M)$ are simultaneously finite.
This is proved as in~\cite[(3.2)]{iyengar:golh} using~\cite[(6.5)]{christensen:scatac}
and~\cite[(4.4)]{frankild:sdcms}.
It shows that the finiteness of
$\gcdim_{\vf}(M)$ does not depend on the choice of Cohen factorization.  
\end{subprops}

\begin{subprops} \label{sri03}
Arguing as in~\cite[(3.4.1)]{iyengar:golh}, one concludes that the quantities
$\gcdim_{\vf}(X)$, $\gcdim_{\grave{\vf}}(\comp{S}\lotimes_S X)$, and
$\gkdim{\comp{R}\lotimes_RC}_{\comp{\vf}}(\comp{S}\lotimes_S X)$
are simultaneously finite.
\end{subprops}

\begin{subprops} \label{sri07}
Let $D^{\comp{R}}$ and $D^{R'}$ be dualizing complexes for $\comp{R}$ and $R'$,
respectively.
The following conditions are equivalent.
\begin{enumerate}[{\quad\rm(i)}]
\item $\gcdim_{\vf}(M)<\infty$.
\item $\gkdim{R'\lotimes_R C}_{R'}(\comp{S}\lotimes_S M)<\infty$.
\item $\comp{S}\lotimes_S M$ is in $\cata_{\rhom_{R'}(R'\lotimes_R C,D^{R'})}(R')$.
\item $\comp{S}\lotimes_S M$ is in 
$\cata_{\rhom_{\comp{R}}(\comp{R}\lotimes_R C,D^{\comp{R}})}(\comp{R})$.
\end{enumerate}
When $R$ possesses a dualizing complex $D^R$, these conditions are equivalent to:
\begin{enumerate}[{\quad\rm(v)}]
\item $M$ is in $\cata_{\rhom_R(C,D^R)}(R)$.
\end{enumerate}
This is proved as in~\cite[(3.6)]{iyengar:golh} using Remarks~\ref{lwc01} and~\ref{sdm02''}.
\end{subprops}

\begin{subprops} \label{sri09}
If $R$ admits a dualizing complex $D$, then~\cite[(5.1)]{christensen:scatac}
and~\eqref{sri07} show that
$\gcdim(\vf)$ is finite if and only if $S\lotimes_R\rhom_R(C,D)$
is a semidualizing $S$-complex.
\end{subprops}

\begin{subprops} \label{sri10}
If $D$ is dualizing for $R$, then $\gkdim{D}_{\vf}(M)$ is finite.
Indeed, because $\dot\vf$ is flat with Gorenstein closed fibre,
the complex $R'\lotimes_R D$ is dualizing for $R'$ by Remark~\ref{sdc02},
so the desired conclusion follows from the definition using
Remark~\ref{sdc04}.
\end{subprops}
\end{properties}

We continue with complete intersection dimensions.
When $M$ is a module, Definition~\ref{cidim03} is from~\cite{sahandi:hfd},
which is in turn modeled on~\cite{avramov:cid}. 

\begin{defn} \label{cidim01}
Let $R$ be a local ring.  A \emph{quasi-deformation} of $R$ is a diagram
of local ring homomorphisms $R\xra{\vf} R'\xla{\tau} Q$ such that $\vf$ is
flat, and $\tau$ is surjective with kernel generated by a $Q$-regular sequence;
if the kernel of $\tau$ is generated by a $Q$-regular sequence of length $c$,
we will sometimes say that the quasi-deformation has \emph{codimension $c$}.
\end{defn}

\begin{defn} \label{cidim03}
Let $(R,\m)$ be a local ring.  For each homologically bounded $R$-complex $M$,
define
the 
\emph{complete intersection projective dimension},
\emph{complete intersection flat dimension} and
\emph{complete intersection injective dimension} of $M$ as, respectively,
\begin{align*}
\cipd_R(M)
&:=\inf\left\{\pd_Q(R'\lotimes_R M)-\pd_Q(R')\left| \text{
\begin{tabular}{@{}c@{}}
$R\to R'\from Q$ is a \\ quasi-deformation 
\end{tabular}
}\!\!\!\right. \right\} \\
\cifd_R(M)
&:=\inf\left\{\fd_Q(R'\lotimes_R M)-\pd_Q(R')\left| \text{
\begin{tabular}{@{}c@{}}
$R\to R'\from Q$ is a \\ quasi-deformation 
\end{tabular}
}\!\!\!\right. \right\} \\
\ciid_R(M)
&:=\inf\left\{\id_Q(R'\lotimes_R M)-\pd_Q(R')
\left| \text{
\begin{tabular}{@{}c@{}}
$R\to R'\from Q$ is a \\ quasi-deformation 
\end{tabular}
}\!\!\!\right. \right\}.
\end{align*}
When $M$ is homologically finite we follow~\cite{avramov:cid,sather:cidc} and 
define the \emph{complete intersection dimension} of $M$ as
$\cidim_R(M):=\cipd_R(M)$.
\end{defn}

\begin{disc} \label{cidim04}
Let $R$ be a local ring and $M$ a homologically bounded $R$-complex.

Given a quasi-deformation $R\to R'\from Q$, a result of Gruson and 
Raynaud~\cite[Seconde Partie, Thm.~(3.2.6)]{raynaud:cpptpm}, and 
Jensen~\cite[Prop.~6]{jensen:vl} tells us that the quantities
$\pd_Q(R'\lotimes_RM)$ and $\fd_Q(R'\lotimes_RM)$ are simultaneously
finite.  From this, it follows that
$\cifd_R(M)<\infty$ if and only if $\cipd_R(M)<\infty$.

Arguing as in~\cite[(1.13.2)]{avramov:cid} and using~\cite[(4.2)]{avramov:hdouc},
one deduces that the
quantities $\cifd_R(M)<\infty$ if  and only if $\cifd_{\comp{R}}(\comp{R}\lotimes_R M)<\infty$.
On the other hand, we do not know if $\ciid_R(M)<\infty$ implies 
$\ciid_{\comp{R}}(\comp{R}\lotimes_R M)<\infty$.
See, however, Corollary~\ref{ciid02}\eqref{ciid02item2}.
\end{disc}

We make the next definition 
for our version of Foxby equivalence in Theorem~\ref{thmC}.

\begin{defn} \label{ci*id01}
Let $(R,\m)$ be a local ring and $M$ a homologically bounded $R$-complex.
Define the \emph{upper complete intersection injective dimension} of $M$ as
$$
\uciid_R(M)
:=\inf\left\{
\text{\begin{tabular}{@{}r@{}}
$\id_Q(R'\lotimes_R M)$ \\
$-\pd_Q(R')$ 
\end{tabular}}\left| \text{
\begin{tabular}{@{}c@{}}
$R\to R'\from Q$ is a quasi-deformation \\
such that $R'$ has Gorenstein formal \\
fibres and $R'/\m R'$ is Gorenstein
\end{tabular}
}\!\!\!\right. \right\}.
$$
\end{defn}

\begin{disc} \label{uciid01}
The formal fibres of $R'$ are Gorenstein when $R'$ is complete or, more generally,
when $R'$ is excellent or admits a dualizing complex.
\end{disc}

\begin{disc} \label{ci*id02}
Let $R$ be a local ring and $M$ a homologically bounded $R$-complex.
If $\uciid_R(M)<\infty$, then  the definitions imply 
$\ciid_R(M)<\infty$.
\end{disc}

\begin{question} \label{ci*id03}
Let $R$ be a local ring and $M$ a homologically bounded $R$-complex.
If $\ciid_R(M)<\infty$, 
must we have $\uciid_R(M)<\infty$?
\end{question}

\begin{disc} \label{cidim12}
Assume that $R$ is a complete intersection
and $M$ is a homologically
bounded $R$-complex.
Arguing as in~\cite[(1.3)]{avramov:cid},
we conclude $\cipd_R(M)<\infty$,
and similarly for $\cifd_R(M)$, $\ciid_R(M)$ and $\uciid_R(M)$.
\end{disc}

The following formulas are complete intersection versions of the
classical Bass formula for injective dimension.  
Note that the 
analogue of the Auslander-Buchsbaum formula for complete intersection
dimension was proved in~\cite[(3.3)]{sather:cidc}.

\begin{prop} \label{bass01}
Let $R$ be a local ring and $M$ a homologically finite $R$-complex.
\begin{enumerate}[\quad\rm(a)]
\item \label{bass01b}
If $\ciid_R(M)<\infty$, then $\ciid_R(M)=\depth(R)-\inf(M)$.
\item \label{bass01c}
If $\uciid_R(M)<\infty$, then $\uciid_R(M)=\depth(R)-\inf(M)$.
\end{enumerate}
\end{prop}

\begin{proof}
\eqref{bass01b}
We begin by observing the equality
$$
\ciid_R(M)
=\inf\left\{
\text{\begin{tabular}{@{}r@{}}
$\id_Q(R'\lotimes_R M)$ \\
$-\pd_Q(R')$ 
\end{tabular}}\left| \text{
\begin{tabular}{@{}c@{}}
$R\to R'\from Q$ is a quasi-deformation \\
such that $R'/\m R'$ is artinian
\end{tabular}
}\!\!\!\right. \right\}
$$
where $\m$ is the maximal ideal of $R$.
Indeed, the inequality ``$\leq$'' follows from the fact that the collection of
all quasi-deformations contains the collection of all quasi-deformations
$R\to R'\from Q$ such that $R'/\m R'$ is artinian. For the opposite inequality,
fix a codimension $c$ quasi-deformation
$R\to R'\from Q$.
Choose a prime ideal $P\in\Min_{R'}(R'/\m R')$ and set $\p=\tau^{-1}(P)$.
From~\cite[(5.1.I)]{avramov:hdouc} we conclude  
$$\id_{Q_{\p}}(R'_P\lotimes_R M)=\id_{Q_{\p}}((R'\lotimes_R M)_{\p})
\leq \id_Q(R'\lotimes_R M).$$
One checks readily that the localized
diagram $R\xra{\vf_P} R'_P\xla{\tau_P} Q_{\p}$ is a quasi-deformation
and that the closed fibre $R'_P/\m R'_P\cong(R'/\m R')_P$ is artinian.
With the equalities $\pd_Q(R')=c=\pd_{Q_{\p}}(R'_P)$, this establishes 
the other inequality.

Assume $\ciid_R(M)<\infty$ and fix a quasi-deformation
$R\xra{\vf} R'\from Q$ such that 
$R'/\m R'$ is artinian and
$\id_Q(R'\lotimes_R M)<\infty$.
The $Q$-complex $R'\lotimes_R M$ is homologically finite,
so the Bass formula and Auslander-Buchsbaum formulas
provide the first equality in the following sequence:
\begin{align*}
\id_Q(R'\lotimes_R M)-\pd_Q(R') 
&=[\depth(Q)-\inf(R'\lotimes_R M)]-[\depth(Q)-\depth(R')] \\
&=-\inf(M)+\depth(R') \\
&=\depth(R)-\inf(M).
\end{align*}
The second and third equalities follow from the fact that $\vf$ is flat and local 
with artinian closed fibre.
Since this quantity is independent of the 
choice of quasi-deformation, only depending on the finiteness of $\id_Q(R'\lotimes_R M)$,
the advertised formula follows.

Part~\eqref{bass01c}
is established similarly.
\end{proof}

We close this section with relative versions of 
complete intersection dimensions.

\begin{defn} \label{cidim07}
Let $\vf\colon R\to S$ be a local ring homomorphism and $M$ a homologically finite $S$-complex.
The \emph{complete intersection dimension of $M$ over $\vf$}  
and \emph{complete intersection injective dimension of $M$ over $\vf$}  
are, respectively, 
\begin{align*}
\cidim_{\vf}(M)&:= 
\inf\left\{\cidim_{R'}(\comp{S}\lotimes_S M)-\edim(\Dot{\vf})
\left| \text{\begin{tabular}{c}
$R\xra{\dot\vf}R'\xra{\vf'}\comp{S}$ is a Cohen \\
factorization of $\grave\vf$
\end{tabular}}\right.\!\!\!\right\}\\
\ciid_{\vf}(M)&:= 
\inf\left\{\ciid_{R'}(\comp{S}\lotimes_S M)-\edim(\Dot{\vf})
\left| \text{\begin{tabular}{c}
$R\xra{\dot\vf}R'\xra{\vf'}\comp{S}$ is a Cohen \\
factorization of $\grave\vf$
\end{tabular}}\right.\!\!\!\right\}.
\end{align*}
The \emph{complete intersection dimension of $\vf$} and
\emph{complete intersection injective dimension of $\vf$} are 
$\cidim(\vf):=\cidim_{\vf}(S)$ and $\ciid(\vf):=\ciid_{\vf}(S)$.
\end{defn}

\begin{disc} \label{cidim11}
We do not introduce $\cipd_{\vf}(M)$ and $\cifd_{\vf}(M)$ because
they would be the same as $\cidim_{\vf}(M)$ by Remark~\ref{cidim04}.
On the other hand, we do not introduce $\uciid_{\vf}(M)$ because we do not need it for our
results.
\end{disc}

\begin{properties} \label{cidim0}
Fix a local ring homomorphism $\vf\colon R\to S$, a Cohen factorization $R\to R'\to \comp{S}$
of $\grave{\vf}$, and a homologically finite $S$-complex $M$.

\begin{subprops}  \label{cidim08}
If $R$ is a
complete intersection, then $\cidim_{\vf}(M)$ and $\ciid_{\vf}(M)$
are finite.
This  
follows from Remark~\ref{cidim12}
because the
fact that $\dot\vf$ is flat with regular closed fibre implies that $R'$ is a
complete intersection.
\end{subprops}

\begin{subprops} \label{cidim09}
As in~\eqref{sri03} one checks that the quantities
$\cidim_{\vf}(M)$, $\cidim_{\grave\vf}(\comp{S}\lotimes_SM)$ and
$\cidim_{\comp\vf}(\comp{S}\lotimes_SM)$ are simultaneously finite,
as are the quantities
$\ciid_{\vf}(M)$, 
$\ciid_{\grave\vf}(\comp{S}\lotimes_SM)$ and
$\ciid_{\comp\vf}(\comp{S}\lotimes_SM)$.
\end{subprops}
\end{properties}

\begin{questions} \label{cidim10}
Let $\vf\colon R\to S$ be a local ring homomorphism and $M$
a homologically finite $S$-complex. 
Is the finiteness of $\cidim_{\vf}(M)$ and/or $\ciid_{\vf}(M)$ independent of the choice of 
Cohen factorization?
Are  $\cidim_{\vf}(M)$ and $\cifd_R(M)$ simultaneously finite?
Are  $\ciid_{\vf}(M)$ and $\ciid_R(M)$ simultaneously finite?
\end{questions}

\section{Structure of Quasi-Deformations} \label{sec5}

Given a quasi-deformation of $R$, we show in this section
how to construct nicer quasi-deformations that are related to the original one.
One outcome  is the proof (in~\ref{qd02}) of Theorem~\ref{thmF} from the introduction.
This is similar in spirit to~\cite[(1.14)]{avramov:cid}, but different in scope,
and will be used frequently in the sequel. 

\begin{lem} \label{qd01}
Let $R\xra{\vf}R'\xla{\tau}Q$ be a codimension $c$
quasi-deformation
of the local ring $(R,\m)$ such that $R'/\m R'$ is Cohen-Macaulay.  
There exists a commutative
diagram of local ring homomorphisms
$$\xymatrix{
R\ar[r]^-{\vf}\ar[rd]_-{\vf'} & R' \ar[d]^f & Q \ar[l]_-{\tau} \ar[d]^g \\
& R'' & Q' \ar[l]_-{\tau'}
}$$
such that $\vf'$ is flat with Gorenstein closed fibre,
$\tau'$ is surjective with kernel generated by  $Q'$-regular sequence of length $c$,
the natural map $R'\otimes_QQ'\to R''$ is bijective,
and one has $\tor^Q_{\geq 1}(R',Q')=0$ and $\dim(R''/\m R'')=\dim(R'/\m R')$.
\end{lem}

\begin{proof}
Fix  a Cohen factorization $R\to S\to \comp{R'}$ of the 
semi-completion $\grave\vf\colon R\to\comp{R'}$.
Because $\vf$ is flat with Cohen-Macaulay closed fibre,
we know from~\cite[(3.8.1),(3.8.3)]{avramov:solh}
that $\comp{R'}$ is perfect as an $S$-module, say, of grade $g$.
This yields $\ext^n_S(\comp{R'},S)=0$ for each $n\neq g$,
so there is an isomorphism $\rhom_S(\comp{R'},S)\simeq\shift^{-g}\ext^g_S(\comp{R'},S)$.
(In the language of~\cite{avramov:rhafgd}
this complex  is \emph{dualizing for $\vf$}.
See~\cite{avramov:rhafgd,frankild:qcmpolh} for more
properties and
applications.)  In particular,
the $\comp{R'}$-complex
$\rhom_S(\comp{R'},S)$ 
is semidualizing  by~\cite[(6.6)]{christensen:scatac}, 
and so the module
$C=\ext^g_S(\comp{R'},S)$ is a semidualizing $\comp{R'}$-module.  
Hence, we have from the definition  $\ext^{\geq 1}_{\comp{R'}}(C,C)=0$.

We claim that $C$ is flat as an $R$-module.  
The module $C$ is finitely generated over $\comp{R'}$, and the 
semi-completion $\grave\vf\colon R\to\comp{R'}$ is a local ring homomorphism.
It follows easily that $\inf(k\lotimes_R C)=0$.  
Also, we have an equality $\fd_R(C)=\sup(k\lotimes_R C)$ from~\cite[(5.5.F)]{avramov:hdouc},
and so it suffices to show
$\amp(k\lotimes_R C)=0$.  From~\cite[(5.10)]{avramov:rhafgd} we know that
$k\lotimes_R C$ is dualizing for  $\comp{R'}/\m\comp{R'}$.  
Since the ring $\comp{R'}/\m\comp{R'}$ is Cohen-Macaulay, this implies
$\amp(k\lotimes_R C)=0$ by~\cite[(3.7)]{christensen:scatac} as desired.

The ring $\comp{Q}$ is complete, and the local ring homomorphism
$\comp{\tau}\colon\comp{Q}\to\comp{R'}$ is surjective with kernel
generated by a $\comp{Q}$-regular sequence $\x$ of length $c$.
Using~\cite[(1.7)]{auslander:lawlom},
the vanishing $\ext^{2}_{\comp{R'}}(C,C)=0$
implies that there is a finitely generated $\comp{Q}$-module
$B$ such that $\x$ is $B$-regular and
$C\cong\comp{R'}\otimes_{\comp Q}B\cong B/\x B$.
Let $\tau_1\colon B\to B/\x B\cong C$ denote the natural surjection.

Now we construct the desired diagram guided by the following.
$$\xymatrix{
R\ar[r]^-{\vf}\ar[rd]_-{\grave\vf} & R' \ar[d]^-{\varepsilon_{R'}} & Q \ar[l]_-{\tau} \ar[d]^-{\varepsilon_{Q}} \\
& \comp{R'}\ar[d]^-{f_1} & \comp{Q} \ar[l]_-{\comp\tau}\ar[d]^-{g_1} \\
& \comp{R'}\ltimes C & \comp{Q}\ltimes B \ar[l]_-{\ \ \tau'}
}$$
The triangle in the diagram commutes by
definition, and the upper square commutes by~\ref{basics01}.
One checks readily that the map 
$\tau':=\comp{\tau}\ltimes\tau_1\colon \comp{Q}\ltimes B\to \comp{R'}\ltimes C$ 
is a local ring homomorphism making the bottom square commute.

The fact that the sequence $\x$ is $\comp{Q}$-regular
and $B$-regular implies 
$$\tor^Q_{\geq 1}(R',\comp{Q}\ltimes B)\cong\tor^Q_{\geq 1}(Q/\x Q,\comp{Q}\ltimes B)=0.$$
Using the surjectivity of $\tau$ one checks that the natural map 
$R'\otimes_Q\comp{Q}\to \comp{R'}$ is bijective.
The construction of $\tau'$
shows that $\Ker(\tau')$ is generated by the $(\comp{Q}\ltimes B)$-regular
sequence $g(\x)$
and 
the natural map 
$\comp{R'}\otimes_{\comp{Q}}(\comp{Q}\ltimes B)\to \comp{R'}\ltimes C$ is bijective.
This implies that 
the natural map $R'\otimes_Q(\comp{Q}\ltimes B)\to \comp{R'}\ltimes C$ is bijective.

As $\comp{R'}$ and $C$ are both $R$-flat, we see that the composition
$\vf'=f_1\circ\grave\vf=f_1\circ\varepsilon_{R'}\circ\vf$ is flat.  Also, the closed fibre
of $\vf'$ is $\comp{R'}/\m\comp{R'}\ltimes C/\m C$.
Because
$k\lotimes_R C\simeq C/\m C$ is dualizing for $k\otimes_R\comp{R'}$, 
we conclude that the closed fibre of $\vf'$ is  Gorenstein.
Thus, setting $R''=\comp{R'}\ltimes C$ and $Q'=\comp{Q}\ltimes B$
yields the desired diagram.
\end{proof}

\begin{para} \label{qd02}
\emph{Proof of Theorem~\ref{thmF}.}
One implication follows immediately from the definition of $\cifd_R(M)$.  
For the other implication, assume $\cifd_R(M)<\infty$
and fix a
quasi-deformation
$R\xra{\vf}R'\xla{\tau}Q$ 
such that $\fd_{Q}(R'\lotimes_RM)<\infty$.
Choose a prime ideal $P\in\Min_{R'}(R'/\m R')$ and set $\p=\tau^{-1}(P)$.
From~\cite[(5.1.F)]{avramov:hdouc} we conclude  
$$\fd_{Q_{\p}}(R'_P\lotimes_R M)=\fd_{Q_{\p}}((R'\lotimes_R M)_{\p})<\infty.$$
One checks readily that the localized
diagram $R\xra{\vf_P} R'_P\xla{\tau_P} Q_{\p}$ is a quasi-deformation
and that the closed fibre $R'_P/\m R'_P\cong(R'/\m R')_P$ is artinian.
Thus, we may replace the original quasi-deformation
with the localized one
in order to assume that the closed fibre $R'/\m R'$ is artinian
and, hence, Cohen-Macaulay.

Lemma~\ref{qd01} now yields
a commutative
diagram of local ring homomorphisms
$$\xymatrix{
R\ar[r]^-{\vf}\ar[rd]_-{\vf'} & R' \ar[d]^f & Q \ar[l]_-{\tau} \ar[d]^g \\
& R'' & Q' \ar[l]_-{\tau'}
}$$
such that $\vf'$ is flat with Gorenstein closed fibre,
$\tau'$ is surjective with kernel generated by  $Q$-regular sequence,
the natural map $Q'\otimes_QR'\to R''$ is bijective,
and one has $\tor^Q_{\geq 1}(Q',R')=0$ and $\dim(R''/\m R'')=\dim(R'/\m R')=0$.
It follows that the diagram $R\xra{\vf'}R''\xla{\tau'}Q'$ is a quasi-deformation,
and so it suffices to show that $\fd_{Q'}(R''\lotimes_RM)$ is finite.

Note that the conditions 
$Q'\otimes_QR'\cong R''$ and $\tor^Q_{\geq 1}(R',Q')=0$
yield an isomorphism
$Q'\lotimes_QR'\simeq R''$.
This yields the second equality in the following sequence:
\begin{align*}
\fd_{Q'}(R''\lotimes_RM)
&=\fd_{Q'}(R''\lotimes_{R'}(R'\lotimes_RM)) \\
&=\fd_{Q'}((Q'\lotimes_QR')\lotimes_{R'}(R'\lotimes_RM)) \\
&=\fd_{Q'}(Q'\lotimes_{Q}(R'\lotimes_RM)) \\
&\leq \fd_Q(R'\lotimes_RM) \\
&<\infty.
\end{align*}
The first equality follows from the commutativity of the displayed diagram.
The third equality is tensor-cancellation.
The first inequality is from~\cite[(4.2.F)]{avramov:hdouc},
and the second inequality is by assumption. \qed
\end{para}

\begin{disc} \label{qd03}
We do not know whether there is a result like Theorem~\ref{thmF}
for $\ciid$.  In fact, it is not even clear if, given an $R$-complex
of finite complete intersection injective dimension, there
exists a quasi-deformation $R\to R'\from Q$ such that $Q$ is complete
and $\id_Q(R'\lotimes_R M)$ is finite.
The place where the proof breaks down is in the final displayed
sequence: it is not true in general that the finiteness of
$\id_Q(R'\lotimes_RM)$ implies 
$\id_{Q'}(Q'\lotimes_{Q}(R'\lotimes_RM))<\infty$.  
However, see Proposition~\ref{qd05}.
\end{disc}

Before proving a version of Theorem~\ref{thmF} for upper complete intersection dimension,
we require the following extension of a result of Foxby~\cite[Thm.~1]{foxby:imufbc} for complexes. 
We shall apply it to the completion homomorphism 
$\varepsilon^{}_Q\colon Q\to \comp{Q}$.
Recall that an $R$-complex $M$ is \emph{minimal} if, for
every homotopy equivalence $\alpha\colon M\to M$, each map $\alpha_i\colon M_i\to M_i$ is 
bijective.

\begin{lem} \label{id01}
Let $Q\to Q'$ be a flat local ring homomorphism 
and $N$ a homologically bounded $Q$-complex
such that, for every prime $\p\in\spec(Q)$ such that $N_{\p}\not\simeq 0$, the
fibre $Q'\otimes_Q (Q_{\p}/\p Q_{\p})$ is Gorenstein.
If $\id_Q(N)<\infty$, then $\id_{\comp{Q}}(\comp Q\lotimes_Q N)<\infty$.
\end{lem}

\begin{proof}
Within this proof, we work in the category of $Q$-complexes, as opposed to the derived
category.  In particular,  a morphism of $Q$-complexes $f\colon X\to Y$ 
is a ``quasiisomorphism'' if each map $\HH_n(f)\colon \HH_n(X)\to \HH_n(Y)$
is bijective, and $f$ is an ``isomorphism'' if each map $f_n\colon X_n\to Y_n$ is
bijective.  Quasiisomorphisms are identified by the symbol $\simeq$, and 
isomorphisms are identified by the symbol $\cong$.

Set $j=\id_Q(N)$ and
let $N\simeq I$ be an injective resolution over $Q$ such that $I_n=0$
for all $n<-j$.
From~\cite[(12.2.2)]{avramov:dgha} there exist $Q$-complexes $I'$ and $I''$ such that
$I'$ is minimal and such that $I''\simeq 0$ and $I\cong I'\oplus I''$.
It follows that $I'$ is a bounded complex of injective $Q$-modules
such that $N\simeq I'$, so
we may replace $I$ with $I'$ in order to assume that $I$ is minimal.  

\textbf{Claim.} Given a prime $\q\in\spec(Q)$ such that $N_{\q}\simeq 0$, we have
$I_{\q}= 0$.  
To prove this,  we start with the quasiisomorphisms $I_{\q}\simeq N_{\q}\simeq 0$ and
consider  the natural morphism $g\colon I\to I_{\q}$.
For each integer $n$, 
the module $I_n$ is isomorphic to a direct sum of injective hulls
$E_Q(Q/\p)$.  Given the isomorphism of $Q$-modules
$$E_Q(Q/\p)_{\q} 
\begin{cases}
=0 & \text{if $\p\not\subseteq\q$} \\
\cong E_Q(Q/\p) & \text{if $\p\subseteq\q$}
\end{cases}
$$
it follows readily
that  each $g_n$ is a split surjection. 
From this we conclude 
that the complex $J=\Ker(g)$ is a bounded complex of injective $Q$-modules.
Furthermore, we use the quasiisomorphism $I_{\q}\simeq 0$ in
the long exact sequence associated to  the exact sequence
$0\to J\xra{h} I\xra{g} I_{\q}\to 0$ in order to conclude that
$h$ is a quasiisomorphism. 
Because $J$ and $I$ are bounded complexes of injective $Q$-modules,
it follows that $h$ is a homotopy equivalence.
Now, using~\cite[(1.7.1)]{avramov:aratc}
we conclude that $h$ is surjective and so $I_{\q}=0$.

From~\cite[Thm.~1]{foxby:imufbc} the fact that the formal
fibre $Q'\otimes_Q (Q_{\p}/\p Q_{\p})$ is Gorenstein
whenever $N_{\p}\not\simeq 0$
implies that each module $Q'\otimes_Q I_n$
has finite injective dimension over $Q'$.  Hence,
the complex 
$Q'\otimes_Q I$ is a bounded complex of $Q'$-modules
of finite injective dimension.  
From the quasiisomorphism
$Q'\lotimes_Q N\simeq Q'\otimes_Q I$
we conclude
$\id_{Q'}(Q'\lotimes_Q N)=\id_{Q'}(Q'\otimes_Q I)<\infty$.
\end{proof}

In the conclusion of the next result,
notice that $R''$ is complete,
and hence it has Gorenstein formal fibres.

\begin{prop} \label{qd04}
Let $(R,\m)$ be a local ring and $M$ a homologically bounded $R$-complex.
Then $\uciid_R(M)<\infty$ if and only if there exists a quasi-deformation
$R\to R''\from Q'$ such that $Q'$ is complete, the closed fibre $R''/\m R''$ is artinian,
and Gorenstein, and $\id_{Q'}(R''\lotimes_RM)<\infty$.
\end{prop}

\begin{proof}
For the nontrivial implication, assume $\uciid_R(M)<\infty$, 
and fix  a
quasi-deformation
$R\xra{\vf}R'\xla{\tau}Q$ 
such that $R'$ has Gorenstein formal fibres, $R'/\m R'$ is Gorenstein, and
$\id_{Q}(R'\lotimes_RM)<\infty$.
Choose a prime ideal $P\in\Min_{R'}(R'/\m R')$ and set $\p=\tau^{-1}(P)$.
We conclude from~\cite[(5.1.I)]{avramov:hdouc} that the quantity
$$\id_{Q_{\p}}(R'_P\lotimes_R M)=\id_{Q_{\p}}((R'\lotimes_R M)_{\p})$$
is finite.  
Also,  the ring $R'_P$ has Gorenstein formal fibres because $R'$ has the same.
Because the ring $R'_P/\m R'_P\cong (R'/\m R')_P$ is artinian and Gorenstein
the proof of Theorem~\ref{thmF}
shows that we may replace the original quasi-deformation
with a localized one
in order to assume that the closed fibre $R'/\m R'$ is artinian.

Completing our quasi-deformation yields
the next commutative
diagram of local ring homomorphisms
$$\xymatrix{
R\ar[r]^-{\vf}\ar[rd]_-{\grave\vf} & R' \ar[d]^{\varepsilon_{R'}} & Q \ar[l]_-{\tau} \ar[d]^{\varepsilon_{Q}} \\
& \comp{R'} & \comp{Q}. \ar[l]_-{\comp{\tau}}
}$$
As in the proof of Theorem~\ref{thmF} 
it suffices to show that the quantity
$$\id_{\comp Q}(\comp{R'}\lotimes_RM)=\id_{\comp Q}(\comp Q\lotimes_{Q}(R'\lotimes_RM))$$ 
is finite.  
For each $\p\in \spec(Q)$ such that 
$(R'\lotimes_RM)_{\p}\not\simeq 0$, we have
$\p\supseteq\Ker(\tau)$ because
$R'\lotimes_RM$ is an $R'$-complex and $\tau$ is surjective.
For each such $\p$, set $\p'=\p R'\in\spec(R')$.
The fact that $\tau$ is surjective yields an isomorphism of closed fibres
$\comp{Q}\otimes_Q (Q_{\p}/\p Q_{\p})
\cong \comp{R'}\otimes_{R'}(R'_{\p'}/\p' R'_{\p'})$.
In particular, each of these rings is Gorenstein by assumption.
Thus, the finiteness of
$\id_{\comp Q}(\comp Q\lotimes_{Q}(R'\lotimes_RM))$
follows from Lemma~\ref{id01}.
\end{proof}

We will see the utility of the next result
in Theorem~\ref{cifdac01'}\eqref{cifdac01'item2'}.

\begin{prop} \label{qd05}
Let $(R,\m)$ be a local ring and $M$  a homologically bounded $R$-complex
with finite length homology.
Then $\ciid_R(M)$ is finite if and only if there exists a quasi-deformation
$R\to R''\from Q'$ such that $Q'$ is complete, the closed fibre $R''/\m R''$ is artinian, 
and $\id_{Q'}(R''\lotimes_RM)$ is finite.
\end{prop}

\begin{proof}
For the nontrivial implication, assume that $\ciid_R(M)$ is finite
and fix a quasi-deformation
$R\xra{\vf}R'\xla{\tau}Q$ 
such that 
$\id_{Q}(R'\lotimes_RM)<\infty$.
Localizing as in the proof of Proposition~\ref{qd04},
we may assume that the closed fibre $R'/\m R'$ is artinian.
Because each $R$-module $\HH_i(M)$ has finite length and
the closed fibre $R'/\m R'$ is artinian, it follows that
each module $\HH_i(R'\lotimes_RM)\cong R'\otimes_R\HH_i(M)$ has finite length
over $R'$ and therefore over $Q$.  In particular, 
there is only one prime $\p\in\spec(Q)$ such that $(R'\lotimes_RM)_{\p}\not\simeq 0$,
namely, the maximal ideal $\p=\n\subset Q$.
Because the formal fibre
$\comp{Q}\otimes_Q Q/\n\cong Q/\n$ is Gorenstein, the argument of Proposition~\ref{qd04}
shows that
$\id_{\comp{Q}}(\comp{R'}\lotimes_R M)$ is finite.
Hence the completed quasi-deformation
$R\xra{\grave\vf}\comp{R'}\xla{\comp\tau}\comp{Q}$ has the desired properties.
\end{proof}

The next result follows from Propositions~\ref{qd04} and~\ref{qd05}
via the proof of~\cite[(1.13)]{avramov:cid}.

\begin{cor} \label{ciid02}
Let $R$ be a local ring and $M$ a homologically bounded $R$-complex.
\begin{enumerate}[\quad\rm(a)]
\item \label{ciid02item1}
One has $\uciid_R(M)<\infty$ if and only if
$\uciid_{\comp{R}}(\comp{R}\lotimes_R M)<\infty$.
\item \label{ciid02item2}
If $M$ has finite length homology,
then one has $\ciid_R(M)<\infty$ if and only if
$\ciid_{\comp{R}}(\comp{R}\lotimes_R M)<\infty$. \qed
\end{enumerate}
\end{cor} 

\section{Stability Results} \label{sec6}

This section documents some situations where 
combinations of complexes either detect or inherit 
homological properties of their component pieces.
We start with partial converses to parts of~\cite[(2.1)]{christensen:apac}
which are useful, e.g., for Theorem~\ref{cifdac01'}\eqref{cifdac01'item2'}.

\begin{prop} \label{sdc06}
Let $R$ be a local ring and $C$ a semidualizing $R$-complex.
Let $M$ and $N\not\simeq 0$ be homologically finite $R$-complexes with
$\pd_R(N)<\infty$.
\begin{enumerate}[\quad\rm(a)]
\item \label{sdc06item2}
If $M\lotimes_RN\in\catac(R)$ or $\rhom_R(N,M)\in\catac(R)$,
then $M\in\catac(R)$.
\item \label{sdc06item4}
If $M\lotimes_RN\in\catbc(R)$ or $\rhom_R(N,M)\in\catbc(R)$,
then $M\in\catbc(R)$.
\end{enumerate}
\end{prop}

\begin{proof}
We prove part~\eqref{sdc06item2}.
The proof of~\eqref{sdc06item4} is similar.

Assume first that $M\lotimes_RN$ is in $\catac(R)$.
This implies that 
$(C\lotimes_R M)\lotimes_R N\simeq C\lotimes_R(M\lotimes_R N)$ is homologically bounded, 
so~\cite[(3.1)]{foxby:daafuc}
implies that $C\lotimes_R M$ is homologically bounded.
(The corresponding implication in the proof 
of~\eqref{sdc06item4} uses tensor-evaluation~\eqref{basics03}.)

We consider the following commutative diagram 
wherein the unmarked isomorphism is tensor-associativity
and $\omega_{CCM}^{}$
is tensor-evaluation~\eqref{basics03}.
$$\xymatrix{
M\lotimes_RN \ar[rr]^-{\gamma_M^C\lotimes_R N} \ar[d]_-{\gamma_{M\lotimes_R N}^C} 
&& (C\lotimes_R\rhom_R(C,M))\lotimes_R N \ar[d]_{\simeq} \\
C\lotimes_R\rhom_R(C,M\lotimes_RN) && C\lotimes_R(\rhom_R(C,M)\lotimes_RN) 
\ar[ll]_{C\lotimes_R\omega^{}_{CCM}}^-{\simeq}
}$$
Because $\gamma_{M\lotimes_R N}^C$ is an isomorphism,
the diagram shows that the same is true of
$\gamma_M^C\lotimes_RN$,
and so~\cite[(2.10)]{iyengar:golh} implies that 
$\gamma_M^C$ is an isomorphism.

Assume next $\rhom_R(N,M)\in\catac(R)$.
Remark~\ref{fpd01} says that $\rhom_R(N,R)$ is a homologically
finite $R$-complex of finite projective dimension
such that 
$$\rhom_R(N,M)\simeq\rhom_R(N,R)\lotimes_R M.$$  
Hence, using $\rhom_R(N,R)$ in place of $N$ in the first
case, we find $M\in\catac(R)$.
\end{proof}

The following result paves the way for Proposition~\ref{sdm03}
which is used in the proof of Theorem~\ref{cifdac01}\eqref{cifdac01item4}.

\begin{lem} \label{sdm04}
Let $R$ be a local ring and $C$ a semidualizing $R$-module.  For each homologically
bounded $R$-complex $M$, the following conditions are equivalent.
\begin{enumerate}[\quad\rm(i)]
\item \label{sdm04item1}
$\gpd_R(M)$ is finite.
\item \label{sdm04item2}
$\gfd_R(M)$ is finite.
\item \label{sdm04item3}
$\comp{R}\lotimes_RM$ is in 
$\cata_{D^{\comp{R}}}(\comp{R})$.
\end{enumerate}
\end{lem}

\begin{proof}
\eqref{sdm04item1}$\iff$\eqref{sdm04item3}.
When $M$ is a module, this is proved in~\cite[(3.5)]{esmkhani:ghdac}.
For the general case, let $F\simeq M$ be an $R$-flat resolution
of $M$.  It follows that 
$\comp{R}\lotimes_{R} F\simeq\comp{R}\lotimes_{R} M$
is an $\comp{R}$-flat resolution of $\comp{R}\lotimes_{R} M$.
Each $F_i$ is a flat $R$-module and is therefore G-flat over $R$.  
Thus, the quantity 
$\gfd_{R}(M)$ is finite if and only if the kernel
$K_i=\Ker(\partial^F_i)$ has finite G-flat dimension over $R$ for 
some (equivalently, every) $i>\sup(M)$.

Fix an integer $i>\sup(M)$ and consider the following truncations of $F$
\begin{align*}
F' &=\qquad 0\to K_i\to F_i\to F_{i-1}\to\cdots \\
F_{\leq i}&=\qquad 0\xra{\hspace{2.5mm}} 0\xra{\hspace{2.5mm}} F_i\to F_{i-1}\to\cdots.
\end{align*}
Our choice of $i$ implies $F'\simeq M$, and so
$\comp{R}\otimes_RF'\simeq \comp{R}\otimes_RM$.
It follows that $\comp{R}\lotimes_RM$ is in 
$\cata_{D^{\comp{R}}}(\comp{R})$ if and only if $\comp{R}\lotimes_RF'$ is in 
$\cata_{D^{\comp{R}}}(\comp{R})$.
The complexes from the display fit into an exact sequence as follows
$$0\to F_{\leq i}\to F'\to\shift^{i+1}K_i\to 0$$
and applying $\comp{R}\otimes_R-$ to this sequence yields a second exact sequence
$$0\to \comp{R}\otimes_RF_{\leq i}\to \comp{R}\otimes_RF'\to\shift^{i+1}\comp{R}\otimes_RK_i\to 0.$$
The complex $\comp{R}\otimes_RF_{\leq i}$ is a bounded complex of $\comp{R}$-flat modules,
and so it is in $\cata_{D^{\comp{R}}}(\comp{R})$ by Remark~\ref{lwc01}.  A standard 
argument using the exact sequence implies that
$\comp{R}\otimes_RF'$ is in $\cata_{D^{\comp{R}}}(\comp{R})$
if and only if $\shift^{i+1}\comp{R}\otimes_RK_i$ is in $\cata_{D^{\comp{R}}}(\comp{R})$.
That is, the complex $\comp{R}\lotimes_RM$ is in 
$\cata_{D^{\comp{R}}}(\comp{R})$
if and only if $\comp{R}\otimes_RK_i$ is in $\cata_{D^{\comp{R}}}(\comp{R})$.

From the first paragraph of this proof, we know that $\gfd_{R}(M)$ is finite if and only if 
$\gfd_R(K_i)$ is finite.  
From~\cite[(3.5)]{esmkhani:ghdac}, we know that
$\gfd_{R}(K_i)$ is finite if and only if 
$\comp{R}\lotimes_RK_i$ is in 
$\cata_{D^{\comp{R}}}(\comp{R})$.
And the second paragraph shows that $\comp{R}\lotimes_RK_i$ is in 
$\cata_{D^{\comp{R}}}(\comp{R})$
if and only if $\comp{R}\otimes_RM$ is in $\cata_{D^{\comp{R}}}(\comp{R})$.
Thus, the equivalence is established.

The proof of \eqref{sdm04item2}$\iff$\eqref{sdm04item3}
is similar using a projective resolution $P\simeq M$.
\end{proof}

\begin{prop} \label{sdm03}
Let $R$ be a local ring and $C$ a semidualizing $R$-module.  For each homologically
bounded $R$-complex $M$, the following conditions are equivalent.
\begin{enumerate}[\quad\rm(i)]
\item \label{sdm03item1}
$\gcpd_R(M)$ is finite.
\item \label{sdm03item2}
$\gcfd_R(M)$ is finite.
\item \label{sdm03item3}
$\comp{R}\lotimes_RM$ is in 
$\cata_{\rhom_{\comp{R}}(\comp{C},D^{\comp{R}})}(\comp{R})$.
\item \label{sdm03item4}
$\gkpd{\comp{C}}_{\comp{R}}(\comp{R}\lotimes_RM)$ is finite.
\item \label{sdm03item5}
$\gkfd{\comp{C}}_{\comp{R}}(\comp{R}\lotimes_RM)$ is finite.
\end{enumerate}
\end{prop}

\begin{proof}
The equivalences
\eqref{sdm03item3}$\iff$\eqref{sdm03item4}$\iff$\eqref{sdm03item5}
are in Remarks~\ref{sdm02} and~\ref{sdm02'}.

\eqref{sdm03item2}$\iff$\eqref{sdm03item5}
We start by showing that $\gfd_{R\ltimes C}(M)$ is finite if and only if
$\comp{R}\lotimes_R M$ is in
$\cata_{D^{\comp{R\ltimes C}}}(\comp{R\ltimes C})$.
One verifies the following isomorphisms readily
$$\comp{R}\cong\comp{R\ltimes C}\otimes_{R\ltimes C}R
\simeq\comp{R\ltimes C}\lotimes_{R\ltimes C}R$$
and from this, we have
$$\comp{R}\lotimes_R M
\simeq(\comp{R\ltimes C}\lotimes_{R\ltimes C}R)\lotimes_R M
\simeq\comp{R\ltimes C}\lotimes_{R\ltimes C} M.
$$
Hence, the desired equivalence follows from Lemma~\ref{sdm04}.

From Remark~\ref{sdm02} we have $\gcfd_R(M)=\gfd_{R\ltimes C}(M)$.
The previous paragraph tells us 
that 
$\gfd_{R\ltimes C}(M)$ is finite if and only if
$\comp{R}\lotimes_R M$ is in
$\cata_{D^{\comp{R\ltimes C}}}(\comp{R\ltimes C})$.
One readily verifies the isomorphism
$\comp{R\ltimes C}\cong\comp{R}\ltimes \comp{C}$
and so $\gcfd_R(M)<\infty$ if and only if 
$\comp{R}\lotimes_R M\in\cata_{D^{\comp{R}\ltimes \comp{C}}}(\comp{R}\ltimes \comp{C})$.
Using Remark~\ref{sdm02'}, we conclude that
$\comp{R}\lotimes_R M$ is in
$\cata_{D^{\comp{R}\ltimes \comp{C}}}(\comp{R}\ltimes \comp{C})$
if and only if 
$\gfd_{\comp{R}\ltimes \comp{C}}(\comp{R}\lotimes_R M)$ is finite,
that is, if and only if
$\gkfd{\comp{C}}_{\comp{R}}(\comp{R}\lotimes_R M)$ is finite.

The equivalence \eqref{sdm03item1}$\iff$\eqref{sdm03item4}
is verified similarly.
\end{proof}

The remaining results of this section deal with stability for complete 
intersection dimensions.  The first one 
is used in the proof of Theorem~\ref{cifdac01'}\eqref{cifdac01'item2'}.

\begin{prop} \label{stable01}
Let $R$ be a local ring and let
$M$ and $N$ be homologically bounded $R$-complexes with $\fd_R(N)<\infty$.
\begin{enumerate}[\quad\rm(a)]
\item \label{stable01item1}
If $\cifd_R(M)<\infty$, then $\cifd_R(M\lotimes_RN)<\infty$.
\item \label{stable01item2}
If $\ciid_R(M)<\infty$, then $\ciid_R(M\lotimes_RN)<\infty$.
\item \label{stable01item3}
If $\uciid_R(M)<\infty$, then $\uciid_R(M\lotimes_RN)<\infty$.
\end{enumerate}
\end{prop}

\begin{proof}
\eqref{stable01item1}
Let $R\to R'\from Q$ be a quasi-deformation such that
$\fd_Q(R'\lotimes_R M)<\infty$.  
The finiteness of $\fd_R(N)$ implies
$\fd_{R'}(R'\lotimes_RN)<\infty$ by~\cite[(4.2)]{avramov:hdouc}
and so~\cite[(4.1.F)]{avramov:hdouc}
provides the finiteness in the next display 
$$\fd_Q(R'\lotimes_R (M\lotimes_RN))
=\fd_Q((R'\lotimes_R M)\lotimes_{R'}(R'\lotimes_RN))<\infty.$$
Hence, we have $\cifd_R(M\lotimes_RN)<\infty$
by definition.

Parts~\eqref{stable01item2} and~\eqref{stable01item3} are proved like~\eqref{stable01item1}
using~\cite[(4.5.F)]{avramov:hdouc}.
\end{proof}

\begin{prop} \label{stable02}
Let $R$ be a local ring with Gorenstein formal fibres. Let
$M$ be a homologically finite $R$-complex
and let $N$ be homologically bounded $R$-complex with $\id_R(N)<\infty$.
\begin{enumerate}[\quad\rm(a)]
\item \label{stable02item1}
If $\cifd_R(M)<\infty$, then $\uciid_R(\rhom_R(M,N))<\infty$.
\item \label{stable02item2}
If $\uciid_R(M)<\infty$, then $\cifd_R(\rhom_R(M,N))<\infty$.
\end{enumerate}
\end{prop}

\begin{proof}
\eqref{stable02item1}
Use Theorem~\ref{thmF} to find  a quasi-deformation 
$R\to R'\from Q$ such that
$R'/\m R'$ is Gorenstein and
$\fd_Q(R'\lotimes_R M)<\infty$. 
Because $R$ has Gorenstein formal fibres, we learn from~\cite[(4.1)]{avramov:glp}
that $R'$ has Gorenstein formal fibres
and, for each prime $\p\in\spec(R)$, the fibre $R'\otimes_R (R_{\p}/\p R_{\p})$
is Gorenstein.  Using~\cite[Thm.~1]{foxby:imufbc} as in the proof of
Proposition~\ref{qd04} we conclude that
$\id_{R'}(R'\lotimes_R N)$ is finite.

Because $M$ is homologically finite and $R'$ is flat over $R$, tensor-evaluation~\eqref{basics03} 
yields the 
first isomorphism in the following sequence:
\begin{align*}
R'\lotimes_R\rhom_R(M,N)
&\simeq \rhom_R(M,R'\lotimes_RN)\\
&\simeq \rhom_R(M,\rhom_{R'}(R',R'\lotimes_RN))\\
&\simeq \rhom_{R'}(R'\lotimes_RM,R'\lotimes_RN).
\end{align*}
The second isomorphism comes from the fact that $R'\lotimes_R N$ is an $R'$-complex,
and the third isomorphism is Hom-tensor adjointness.
This sequence yields the equality in the next sequence
and the finiteness is from~\cite[(4.1.I)]{avramov:hdouc}:
$$\id_Q(R'\lotimes_R\rhom_R(M,N))
=\id_Q(\rhom_{R'}(R'\lotimes_RM,R'\lotimes_RN))<\infty.$$
Hence, we have $\ciid_R(\rhom_R(M,N))<\infty$
by definition.

Part~\eqref{stable02item2} is proved like~\eqref{stable02item1}
using~\cite[(4.5.I)]{avramov:hdouc}.
\end{proof}

The previous result yields the following behavior of
complete intersection dimensions with respect to ``dagger-duality''.

\begin{cor} \label{stable03}
Let $R$ be a local ring admitting a dualizing complex $D$,
and let $M$ be a homologically finite $R$-complex.  
\begin{enumerate}[\quad\rm(a)]
\item \label{stable03item1}
We have $\cifd_R(M)<\infty$ if and only if $\uciid_R(\rhom_R(M,D))<\infty$.
\item \label{stable03item2}
We have $\uciid_R(M)<\infty$ if and only if $\cifd_R(\rhom_R(M,D))<\infty$.
\end{enumerate}
\end{cor}

\begin{proof}
Because $R$ admits a dualizing complex, it has Gorenstein formal fibres
by~\cite[(V.3.1)]{hartshorne:rad}.
So, if $\cifd_R(M)<\infty$, then $\uciid_R(\rhom_R(M,D))<\infty$ by
Proposition~\ref{stable02}\eqref{stable02item1}.
Conversely, if $\uciid_R(\rhom_R(M,D))<\infty$,
then the isomorphism
$$M\simeq\rhom_R(\rhom_R(M,D),D)$$ 
from Remark~\ref{sdc04}
implies
$\cifd_R(M)<\infty$ by Proposition~\ref{stable02}\eqref{stable02item2}.
This establishes part~\eqref{stable03item1}, and
part~\eqref{stable03item2} is proved similarly.
\end{proof}

\begin{prop} \label{stable04}
Let $R$ be a local ring. Let
$M$ be a homologically finite $R$-complex with $\pd_R(M)<\infty$
and let $N$ be homologically bounded $R$-complex.
\begin{enumerate}[\quad\rm(a)]
\item \label{stable04item1}
If $\cifd_R(N)<\infty$, then $\cifd_R(\rhom_R(M,N))<\infty$.
\item \label{stable04item2}
If $\ciid_R(N)<\infty$, then $\ciid_R(\rhom_R(M,N))<\infty$.
\item \label{stable04item3}
If $\uciid_R(N)<\infty$, then $\uciid_R(\rhom_R(M,N))<\infty$.
\end{enumerate}
\end{prop}

\begin{proof}
Remark~\ref{fpd01} says that $\rhom_R(M,R)$ is a homologically
finite $R$-complex of finite projective dimension
such that 
$\rhom_R(N,M)\simeq\rhom_R(N,R)\lotimes_R M$.
Hence, the desired result follows from
Proposition~\ref{stable01}.
\end{proof}

\section{Complete Intersection Dimensions and Foxby Classes} \label{sec4}

The first result of this section contains Theorem~\ref{thmA} from the introduction.

\begin{thm} \label{cifdac01}
Let $R$ be a local ring and fix 
a homologically bounded $R$-complex $M$ and a semidualizing $R$-complex $C$.  
\begin{enumerate}[\quad\rm(a)]
\item \label{cifdac01item1}
If $\cifd_R(M)<\infty$, then
$M\in\catac(R)$.
\item \label{cifdac01item4}
If $C$ is a module and $\cifd_R(M)<\infty$, then
$\gcfd_R(M)<\infty$.
\item \label{cifdac01item3}
If $M$ is homologically finite and $\cidim_R(M)<\infty$, then
$\gkdim{C}_R(M)<\infty$.
\end{enumerate}
\end{thm}

\begin{proof} 
\eqref{cifdac01item1}
Assume $\cifd_R(M)<\infty$, and use Theorem~\ref{thmF} to find
a quasi-deformation $R\to R'\from Q$ such that $Q$ is complete
and $\fd_Q(R'\lotimes_R M)<\infty$.  From~\cite[(4.2)]{frankild:sdcms}
there is a semidualizing $Q$-complex $B$ such that
$R'\lotimes_Q B\simeq R'\lotimes_R C$.
The finiteness of $\fd_Q(R'\lotimes_R M)$ implies 
$R'\lotimes_R M\in\catab(Q)$, 
and so
$R'\lotimes_R M\in\cata_{R'\lotimes_QB}(R')
=\cata_{R'\lotimes_RC}(R')$, and
hence $M\in\catac(R)$; see Remark~\ref{lwc01}.

\eqref{cifdac01item4} and \eqref{cifdac01item3}  
The assumption $\cifd_R(M)<\infty$ implies
$\cifd_{\comp{R}}(\comp{R}\lotimes_RM)<\infty$ by Remark~\ref{cidim04}.
If $D^{\comp{R}}$ is dualizing for $\comp R$,
then part~\eqref{cifdac01item1} implies
that $\comp{R}\lotimes_RM$ is in 
$\cata_{\rhom_{\comp{R}}(\comp{C},D^{\comp{R}})}(\comp{R})$.
When $C$ is a module, Proposition~\ref{sdm03} implies
$\gcfd_R(M)<\infty$.
When $M$ is homologically finite, Remark~\ref{sdm02''}
implies 
$\gkdim{\comp{R}\otimes_R C}(\comp{R}\lotimes_RM)<\infty$, and so
$\gcdim_R(M)<\infty$ by~\cite[(5.10)]{christensen:scatac}.  
\end{proof}

\begin{thm} \label{cifdac01'}
Let $(R,\m,k)$ be a local ring and fix 
a homologically bounded $R$-complex $M$ and a semidualizing $R$-complex $C$.  
\begin{enumerate}[\quad\rm(a)]
\item \label{cifdac01'item2}
If $\uciid_R(M)<\infty$, then
$M\in\catbc(R)$.
\item \label{cifdac01'item2'}
If $M$ is homologically finite and $\ciid_R(M)<\infty$, then
$M\in\catbc(R)$.
\item \label{cifdac01'item5}
Assume that  $R$ admits a dualizing complex $D^R$ and that $C$ is a module.
If either $\uciid_R(M)<\infty$ or 
$M$ is homologically finite and $\ciid_R(M)<\infty$, then
$\gcid_R(M)<\infty$.
\end{enumerate}
\end{thm}

\begin{proof} 
\eqref{cifdac01'item2} This is proved like
Theorem~\ref{cifdac01}\eqref{cifdac01item1} using Proposition~\ref{qd04}.  

\eqref{cifdac01'item2'} 
Let $K$ be the Koszul complex over $R$ on a minimal generating sequence
for the maximal ideal $\m$.  The conditions
$\fd_R(K)<\infty$ and $\ciid_R(M)<\infty$ imply
$\ciid_R(M\lotimes_RK)<\infty$ by Proposition~\ref{stable01}\eqref{stable01item2}.
As the homology $\HH(M\lotimes_RK)$ is a finite dimensional
vector space over $k$, Proposition~\ref{qd05} yields a quasi-deformation
$R\to R''\from Q'$ such that $Q'$ is complete
and $\id_{Q'}(R''\lotimes_R(M\lotimes_RK))$ is finite.
Arguing as in the proof of Theorem~\ref{cifdac01}~\eqref{cifdac01item1} 
we conclude that $M\lotimes_RK$ is in $\catbc(R)$,
and so Proposition~\ref{sdc06}\eqref{sdc06item4}
implies $M\in\catbc(R)$.

\eqref{cifdac01'item5}
By parts~\eqref{cifdac01'item2} and~\eqref{cifdac01'item2'},
the assumptions  imply
that $M$ is in 
$\catb_{\rhom_{R}(C,D^{R})}(R)$, and
from Remark~\ref{sdm02'}, we conclude
$\gcid_R(M)<\infty$.
\end{proof}

\begin{questions} \label{ciidbc01}
Does the conclusion of Theorem~\ref{cifdac01'}\eqref{cifdac01'item2}
also hold if we replace the assumption
$\uciid_R(M)<\infty$ with $\ciid_R(M)<\infty$?
(c.f.~\ref{ci*id03}.)
Does the conclusion of Theorem~\ref{cifdac01'}\eqref{cifdac01'item5}
hold if we do not assume that $R$ admits a dualizing complex?
\end{questions}

\begin{para} \label{foxby01}
\emph{Proof of Theorem~\ref{thmC}.}
We start with the forward implication of~\eqref{thmCitem1}.
Assume $\cifd_R(M)<\infty$ and use Theorem~\ref{thmF}
to find a quasi-deformation $R\xra{\vf} R'\xla{\tau} Q$ such that
$Q$ is complete, $R'/\m R'$ is Gorenstein, and
$\fd_Q(R'\lotimes_R M)<\infty$. 
Note that, since $Q$ is complete, the same  is true of $R'$,
and so $R'$ has Gorenstein formal fibres.
Because $\vf$ is flat with Gorenstein closed fibre,
we know that $R'\lotimes_R D^R$ is dualizing for $R'$
by Remark~\ref{sdc02}.
As $Q$ is complete, it admits a dualizing complex $D^Q$.
Again by Remark~\ref{sdc02},
the fact that $\tau$ is surjective with kernel generated by a $Q$-regular
sequence implies that $R'\lotimes_QD^Q$ is dualizing for $R'$
and so $R'\lotimes_QD^Q\sim R'\lotimes_R D^R$ by~\cite[(V.3.1)]{hartshorne:rad}.  
After replacing $D^Q$ with $\shift^iD^Q$ for an appropriate integer $i$,
we assume without loss of generality $R'\lotimes_RD^Q\simeq R'\lotimes_R D^R$.  

Theorem~\ref{cifdac01}\eqref{cifdac01item1} implies that $M$ is in
$\cata_{D^R}(R)$,
and so $D^R\lotimes_RM$ is homologically bounded.
Because $\fd_Q(R'\lotimes_R M)$ is finite, we know that
$\id_Q(D^Q\lotimes_Q(R'\lotimes_R M))$ is finite as well
by~\cite[(4.5.F)]{avramov:hdouc}.
Hence, the following sequence of isomorphisms
$$D^Q\lotimes_Q(R'\lotimes_R M)
\simeq (D^Q\lotimes_QR')\lotimes_R M 
\simeq (R'\lotimes_R D^R)\lotimes_R M 
\simeq R'\lotimes_R (D^R\lotimes_R M)
$$
yields $\id_Q(R'\lotimes_R (D^R\lotimes_R M))<\infty$.
By definition, we have $\uciid_R(D^R\lotimes_RM)<\infty$.

We continue with the forward implication of~\eqref{thmCitem2}.
Assume $\uciid_R(M)<\infty$ and use Proposition~\ref{qd04}
to find a quasi-deformation $R\xra{\vf} R'\xla{\tau} Q$ such that
$Q$ is complete, $R'/\m R'$ is Gorenstein, and
$\id_Q(R'\lotimes_R M)<\infty$. 
As above, the ring $Q$ admits a dualizing complex $D^Q$
such that $R'\lotimes_RD^Q\simeq R'\lotimes_R D^R$.  

Theorem~\ref{cifdac01}\eqref{cifdac01'item2} implies $M\in\catb_{D^R}(R)$,
and so $\rhom_R(D^R,M)$ is homologically bounded.
As $\id_Q(R'\lotimes_R M)$ is finite, we know that
$\fd_Q(\rhom_Q(D^Q,R'\lotimes_R M))$ is finite as well
by~\cite[(4.5.I)]{avramov:hdouc}.
Hence, the following sequence of isomorphisms
\begin{align*}
\rhom_Q(D^Q,R'\lotimes_R M)
&\simeq\rhom_Q(D^Q,\rhom_{R'}(R',R'\lotimes_R M)) \\
&\simeq\rhom_{R'}(D^Q\lotimes_QR',R'\lotimes_R M) \\
&\simeq\rhom_{R'}(D^R\lotimes_R R',R'\lotimes_R M) \\
&\simeq\rhom_{R}(D^R,\rhom_{R'}(R',R'\lotimes_R M)) \\
&\simeq\rhom_{R}(D^R,R'\lotimes_R M) \\
&\simeq R'\lotimes_R \rhom_{R}(D^R,M)
\end{align*}
yields $\fd_Q(R'\lotimes_R \rhom_{R}(D^R,M))<\infty$
and so $\cifd_R(\rhom_{R}(D^R,M))<\infty$.
The first and fifth isomorphisms are Hom-cancellation;
the second and fourth isomorphisms are Hom-tensor adjointness;
the third isomorphism is from our choice of $D^Q$;
and the last isomorphism is tensor-evaluation~\eqref{basics03}.

For the reverse implication of~\eqref{thmCitem1},
assume  $\uciid_R(D^R\lotimes_RM)<\infty$.
From the forward implication of~\eqref{thmCitem2}
we conclude $\cifd(\rhom_R(D^R,D^R\lotimes_RM)<\infty$.
Theorem~\ref{cifdac01'}\eqref{cifdac01'item2} implies 
$D^R\lotimes_RM\in\catb_{D^R}(R)$,
and so there is an isomorphism
$$M\simeq \rhom(D^R,D^R\lotimes_RM).$$
Thus, we conclude $\cifd_R(M)
=\cifd(\rhom_R(D^R,D^R\lotimes_RM))<\infty$.

The proof of the reverse implication of~\eqref{thmCitem2}
is similar. \qed
\end{para}

\begin{question} \label{foxby02}
In Theorem~\ref{thmC}, can we replace $\uciid$ with $\ciid$? 
(c.f.~\ref{ci*id03}.)
\end{question}

\begin{disc} \label{foxby03}
Foxby equivalence is often described in terms of a diagram,
as in~\cite[(4.2)]{christensen:scatac}.
We show how Theorem~\ref{thmC} adds to this diagram.
Let $R$ be a local ring admitting a dualizing complex $D^R$.
Let $\catcif(R)$ and $\catucii(R)$ denote the full subcategories of $\D_{\mathrm{b}}(R)$
consisting of the complexes $M$ with,
respectively, $\cifd_R(M)<\infty$ and $\uciid_R(M)<\infty$.
Using Theorems~\ref{thmC}, \ref{cifdac01}\eqref{cifdac01item1}
and~\ref{cifdac01'}\eqref{cifdac01'item2} in conjunction with~\cite[(4.2)]{christensen:scatac},
we find that there is a commutative diagram
$$
\xymatrix{
\D(R) \ar@<1ex>[rr]^-{D^R\lotimes_R -} && \D(R) \ar@<1ex>[ll]^-{\rhom_R(D^R,-)} \\
\cata_{D^R}(R) \ar@<1ex>[rr] \ar@{^(->}[u] && \catb_{D^R}(R) \ar@<1ex>[ll] \ar@{^(->}[u] \\
\catcif(R) \ar@<1ex>[rr] \ar@{^(->}[u] && \catucii(R) \ar@<1ex>[ll] \ar@{^(->}[u] \\
\catf(R) \ar@<1ex>[rr] \ar@{^(->}[u] && \cati(R) \ar@<1ex>[ll] \ar@{^(->}[u]
}$$
where the vertical arrows are the natural full embeddings.
\end{disc}

The next result is the special case of Theorem~\ref{thmC}
wherein $R$ is dualizing for $R$.

\begin{cor} \label{foxby04}
Let $R$ be a local Gorenstein ring and $M$ a homologically bounded $R$-complex.
Then $\cifd_R(M)<\infty$ if and only if $\uciid_R(M)<\infty$.\qed
\end{cor}

\section{Complete Intersection Dimensions over Local Homomorphisms} \label{sec7}

We begin this section with relative versions of parts of
Theorems~\ref{cifdac01} and~\ref{cifdac01'}.

\begin{thm} \label{cifdac02}
Let $\vf\colon R\to S$ be a local ring homomorphism and fix 
a homologically finite $S$-complex $M$ and a semidualizing $R$-complex $C$.  
\begin{enumerate}[\quad\rm(a)]
\item \label{cifdac02item1}
If $\cidim_{\vf}(M)<\infty$, then
$M\in\catac(R)$ and $\gkdim{C}_{\vf}(M)<\infty$.
\item \label{cifdac02item2}
If $\ciid_{\vf}(M)<\infty$, then
$M\in\catbc(R)$.
\end{enumerate}
\end{thm}

\begin{proof} 
\eqref{cifdac02item1}
The finiteness of $\cidim_{\vf}(M)$
yields a Cohen factorization 
$R\to R'\to\comp{S}$ of the semi-completion
$\grave{\vf}\colon R\to\comp{S}$ such that
$\cidim_{R'}(\comp{S}\lotimes_S M)<\infty$.
Hence, Theorem~\ref{cifdac01}\eqref{cifdac01item1} guarantees
\begin{equation} \label{cifdac02eq1} \tag{$\dagger$}
\comp{S}\lotimes_S M\in\catab(R') \qquad \text{for each semidualizing
$S$-complex $B$.}
\end{equation}
This yields
$\comp{S}\lotimes_S M\in\cata_{R'\lotimes_R C}(R')$.
Remark~\ref{lwc01} implies
$\comp{S}\lotimes_S M\in\cata_{C}(R)$.
Arguing as in the proof of~\cite[(5.3.a)]{christensen:scatac}
we deduce $M\in\catac(R)$ and hence~\eqref{cifdac02item1}.
The ring $R'$ is complete, and so admits a dualizing complex $D$.
Again using~\eqref{cifdac02eq1}, we conclude
$\comp{S}\lotimes_S M\in\cata_{\rhom_{R'}(R'\lotimes_R C,D)}(R')$
and so $\gcdim_{\vf}(M)<\infty$ by~\eqref{sri07}.

Part~\eqref{cifdac02item2} is proved like part~\eqref{cifdac02item1},
once we note that Theorem~\ref{cifdac01'}\eqref{cifdac01'item2'} applies
because $\comp{S}\lotimes_S M$ is homologically finite over $R'$.  
\end{proof}

Theorem~\ref{thmB} 
is the special case 
$C=R$ and $M=T$ of the next result.

\begin{thm} \label{compose01}
Let $\vf\colon R \to S$ and $\psi\colon S \to T$ be local ring homomorphisms,
and let $C$ be a
semidualizing $R$-complex.  
Fix a homologically finite $T$-complex $M$.
If $\cidim_{\psi}(M)$ and
$\gcdim(\vf)$ are finite, then $\gcdim_{\psi\vf}(M)$ is finite.
\end{thm}

\begin{proof}
\textbf{Step 1.}  
Assume that $R$ is complete, the maps
$\vf$ and $\psi$ are surjective, and the
quantities $\gcdim_R(S)$ and $\cidim_S(M)$ are finite. 
We show $\gcdim_R(M)$ is also finite.
Note that $M$ is homologically finite as an $S$-complex.
Let $D^R$ be a dualizing complex for $R$.
Using~\eqref{sri07} and~\eqref{sri09}, the finiteness of $\gcdim_R(S)$ implies that
$S$ is in $\cata_{\rhom_R(C,D^R)}(R)$
and that 
$S\lotimes_R\rhom_R(C,D^R)$ is a semidualizing $S$-complex.
Using Theorem~\ref{cifdac01}\eqref{cifdac01item1},
the finiteness of $\cidim_S(M)$ implies that
$M$ is in $\cata_{S\lotimes_R\rhom_R(C,D^R)}(S)$.
From Remark~\ref{lwc01} we conclude
$M\in\cata_{\rhom_R(C,D^R)}(R)$, and so Remark~\ref{sdm02''}
implies $\gcdim_R(M)<\infty$.

\textbf{Step 2.}  
We prove the result when the rings $R$, $S$, and $T$ are complete.
Because $T$ is complete,
the finiteness of $\cidim_{\psi}(M)$ provides a Cohen factorization
$$S\xra{\dot\psi}S'\xra{\psi'}T$$ 
such that $\cidim_{S'}(M)<\infty$.
Fix a Cohen factorization 
$$R\xra{\dot\vf}R'\xra{\vf'}S$$ 
of $\vf$.
From~\cite[(1.6)]{avramov:solh} we conclude that there exists
a Cohen factorization 
$$R'\xra{\dot\rho}R''\xra{\rho'}S'$$ 
of $\dot\psi\circ\vf'$ such that $S'\cong R''\otimes_{R'}S$.
The flatness of $R''$ over $R$ implies
\begin{equation} \label{compose01eq1} \tag{$\ddagger$}
S'\cong R''\otimes_{R'}S\simeq R''\lotimes_{R'}S.
\end{equation}
These maps fit into the next commutative diagram of local ring homomorphisms
$$
  \xymatrixrowsep{2.5pc}
  \xymatrixcolsep{2.5pc}
  \xymatrix {
  && R'' \ar@{->}[dr]^{\rho'}  \ar@/^2.5pc/[ddrr]^{\psi'\rho'}  \\
  & R' \ar@{->}[dr]^{\vf'} \ar@{->}[ur]^{\dot\rho} && S'
  \ar@{->}[dr]^{\psi'}  \\
  R \ar@{->}[rr]^{\vf} \ar@/^2.5pc/[uurr]^{\dot\rho \dot\vf}
  \ar@{->}[ur]^{\dot\vf} && S
  \ar@{->}[rr]^{\psi} \ar@{->}[ur]^{\dot\psi} && T }
$$
the outer arcs of which describe a Cohen factorization of the composition
$\psi\circ\vf$.

Let $D^{R'}$ be a dualizing complex for $R'$.
As $\dot\rho$ is flat with
Gorenstein closed fibre, Remark~\ref{sdc02}
implies that $D^{R''}=R''\lotimes_{R'}D^{R'}$
is dualizing for $R''$.
The finiteness of $\gcdim(\vf)$ implies 
$$S\in\cata_{\rhom_{R'}(R'\lotimes_R C,D^{R'})}(R')$$ 
by~\eqref{sri07}.
Hence, we deduce the following membership from Remark~\ref{lwc01}
$$S'\simeq R''\lotimes_{R'}S\in\cata_{R''\lotimes_{R'}\rhom_{R'}(R'\lotimes_R C,D^{R'})}(R')$$ 
while the isomorphism is from~\eqref{compose01eq1}.
The following sequence of isomorphisms
helps us make sense of the semidualizing $R''$-complex
$R''\lotimes_{R'}\rhom_{R'}(R'\lotimes_R C,D^{R'})$:
\begin{align*}
R''\lotimes_{R'}\rhom_{R'}(R'\lotimes_R C,D^{R'})
&\simeq\rhom_{R'}(R'\lotimes_R C,R''\lotimes_{R'}D^{R'}) \\
&\simeq\rhom_{R'}(R'\lotimes_R C,D^{R''}) \\
&\simeq\rhom_{R'}(R'\lotimes_R C,\rhom_{R''}(R'',D^{R''})) \\
&\simeq\rhom_{R''}(R''\lotimes_{R'}R'\lotimes_R C,D^{R''}) \\
&\simeq\rhom_{R''}(R''\lotimes_R C,D^{R''}).
\end{align*}
The first isomorphism is by tensor-evaluation~\eqref{basics03};
the second one is from the definition $D^{R''}=R''\lotimes_{R'}D^{R'}$;
the third one is Hom-cancellation;
the fourth one is Hom-tensor adjointness;
the fifth  one is tensor-cancellation.
This sequence, with the previous display, provides
$$S'\in\cata_{\rhom_{R''}(R''\lotimes_R C,D^{R''})}(R')$$ 
and we conclude 
$\gkdim{R''\lotimes_R C}_{R''}(S')<\infty$ by Remark~\ref{sdm02''}.
Using the condition $\cidim_{S'}(M)<\infty$, Step 1 implies
$\gkdim{R''\lotimes_R C}_{R''}(M)<\infty$.
Because the diagram 
$$R\xra{\dot\rho\circ\dot\vf}R''\xra{\psi'\circ\rho'}T$$
is a Cohen factorization of $\psi \circ\vf$, this implies
$\gcdim_{\psi \circ\vf}(M)<\infty$, as desired.

\textbf{Step 3.}  
We prove the result in general.
The conditions 
$\cidim_{\psi}(M)<\infty$ and
$\gcdim(\vf)<\infty$ imply
$\cidim_{\comp\psi}(\comp{T}\lotimes_TM)<\infty$ and
$\gkdim{\comp{R}\lotimes_R C}(\comp\vf)<\infty$ by~\eqref{sri03} and~\eqref{cidim09}.  
From Step 2, we conclude
$$\gkdim{\comp{R}\lotimes_R C}_{\comp{\psi \circ\vf}}(\comp{T}\lotimes_TM)
=\gkdim{\comp{R}\lotimes_R C}_{\comp\psi \circ\comp\vf}(\comp{T}\lotimes_TM)<\infty$$
and so $\gcdim_{\psi \circ\vf}(M)<\infty$ by~\eqref{sri03}.
\end{proof}

\begin{disc} \label{compose03}
Iyengar and Sather-Wagstaff~\cite[(5.2)]{iyengar:golh} prove the following
decomposition result:  
If $\vf\colon R\to S$ and $\psi\colon S\to T$ are local ring homomorphisms
such that $\psi\circ\vf$ has finite Gorenstein dimension and $\psi$ has finite flat 
dimension, then  $\vf$ has finite Gorenstein dimension.
We do not know whether the conclusion in this result 
holds if we only assume that $\psi$ has finite complete intersection 
dimension.
\end{disc}

\begin{para} \label{ciid01}
\emph{Proof of Theorem~\ref{thmD}.}
For the forward implication, assume $\ciid(\vf)<\infty$.
Consider a Cohen factorization 
$R\xra{\dot\vf} R'\xra{\vf'}\comp S$ of the semi-completion $\grave\vf$
such that
$\ciid_{R'}(\comp S)<\infty$.  This provides
a quasi-deformation
$R'\xra{\sigma}R''\xla{\tau}Q$
such that $\id_Q(R''\otimes_{R'}\comp{S})<\infty$.
The map $\vf'\colon R'\to\comp S$ is a surjective local ring homomorphism,
so the same is true of the base-changed map
$R''\otimes_{R'}\vf'\colon R''\to R''\otimes_{R'}\comp{S}$.
Thus, the composition $(R''\otimes_{R'}\vf')\circ\tau\colon Q\to R''\otimes_{R'}\comp{S}$
is a surjective local ring homomorphism of finite injective dimension.
Using a result of Peskine and Szpiro~\cite[(II.5.5)]{peskine:dpfcl}, we conclude
that $Q$ is Gorenstein.  
Because  $\tau$ is surjective with kernel
generated by a $Q$-regular sequence, this implies that $R''$
is Gorenstein.  Hence,  the flatness of the composition
$\sigma\circ\vf'\colon R\to R''$ implies that $R$ is Gorenstein.
Because $Q$ is Gorenstein, a result of Levin and Vasconcelos~\cite[(2.2)]{levin:hdmr}
says that the finiteness of $\id_Q(R''\otimes_{R'}\comp{S})$
implies $\pd_Q(R''\otimes_{R'}\comp{S})<\infty$.
By definition, we conclude $\cidim(\vf)<\infty$.

For the converse, assume that $R$ is Gorenstein
and $\cifd(\vf)$ is finite.
Fix a Cohen factorization 
$R\xra{\dot\vf} R'\xra{\vf'}\comp S$ of  $\grave\vf$
such that
$\cifd_{R'}(\comp S)<\infty$.  Since $R$ is Gorenstein and $\dot\vf$
is flat with regular closed fibre, we know that $R'$ is Gorenstein.
Remark~\ref{ci*id02} and
Corollary~\ref{foxby04} implies $\ciid_{R'}(\comp S)<\infty$,
and so $\ciid(\vf)<\infty$.
\qed
\end{para}

Theorem~\ref{thmE} from the introduction is contained in Corollary~\ref{ciid03},
which we prove after a definition and a remark.

\begin{defn} \label{contrn01}
A local ring endomorphism $\vf\colon R\to R$ is a 
\emph{contraction} if, for each element $x\in\m$, the sequence
$\{\vf^n(x)\}$ converges to $0$ in the $\m$-adic topology.
\end{defn}

\begin{disc} \label{contrn02}
If $R$ is a local ring of prime characteristic, then 
the Frobenius endomorphism $\vf\colon R\to R$ is a 
contraction.
\end{disc}

\begin{cor} \label{ciid03}
Given a contraction $\vf\colon R\to R$, the 
following are equivalent.
\begin{enumerate}[\quad\rm(i)]
\item \label{ciid03item1}
$R$ is a complete intersection.
\item \label{ciid03item2}
$\ciid(\vf^i)$ is finite for some integer $i\geq 1$.
\item \label{ciid03item3}
$\ciid(\vf^i)$ is finite for each integer $i\geq 1$.
\end{enumerate}
\end{cor}

\begin{proof}
The implication \eqref{ciid03item1}$\implies$\eqref{ciid03item3}
follows from~\eqref{cidim08}, and the implication 
\eqref{ciid03item3}$\implies$\eqref{ciid03item2} is trivial.
For \eqref{ciid03item2}$\implies$\eqref{ciid03item1}, apply
Theorem~\ref{thmD} to conclude that 
$\cidim(\vf^i)$ is finite and then~\cite[(13.5)]{avramov:holh}
implies that $R$ is a complete intersection.
\end{proof}

\begin{cor} \label{ciid04}
Let $\vf\colon R\to S$ and $\psi\colon S\to T$ be local ring homomorphisms,
and let $C$ be a semidualizing $R$-complex.
\begin{enumerate}[\quad\rm(a)]
\item \label{ciid04item1}
If $\vf$ has finite $\text{G}_C$-dimension and $\psi$ has finite complete intersection
injective dimension, then $C$ is dualizing.
\item \label{ciid04item2}
If $\vf$ has finite Gorenstein dimension and $\psi$ has finite complete intersection
injective dimension, then $R$ is Gorenstein.
\end{enumerate}
\end{cor}

\begin{proof}
\eqref{ciid04item1}
The finiteness of $\ciid(\psi)$ implies that $S$ is Gorenstein
by Theorem~\ref{thmD}, and so $\comp{S}$ is Gorenstein.
Let $D$ be a dualizing $\comp{R}$-complex.
Because $\gcdim(\vf)$ is finite
we know that
$\comp S$ is in 
$\cata_{\rhom_{\comp{R}}(\comp{R}\lotimes_R C,D)}(\comp{R})$
by~\eqref{sri07},
and so~\eqref{sri09} implies that the
$\comp S$-complex 
$\comp{S}\lotimes_{\comp R}\rhom_{\comp{R}}(\comp{R}\lotimes_R C,D)$
is semidualizing.
Because $\comp{S}$ is Gorenstein, we have
the following isomorphism from~\cite[(8.6)]{christensen:scatac}.
$$\comp S \sim
\comp{S}\lotimes_{\comp R}\rhom_{\comp{R}}(\comp{R}\lotimes_R C,D)
$$
One concludes readily (using Poincar\'e series, say) that there is an isomorphism
$$\comp R \sim
\rhom_{\comp{R}}(\comp{R}\lotimes_R C,D).$$
Apply the functor $\rhom_{\comp{R}}(-,D)$ to this isomorphism to justify
the second isomorphism in the next sequence:
$$D\simeq\rhom_{\comp{R}}(\comp R,D) \sim
\rhom_{\comp{R}}(\rhom_{\comp{R}}(\comp{R}\lotimes_R C,D),D)
\simeq \comp{R}\lotimes_R C.$$
The first isomorphism is Hom-cancellation, and the third one is 
from Remark~\ref{sdc04}.
Now apply Remark~\ref{sdc02} to conclude that
$C$ is dualizing for $R$, and so~\eqref{sri10}
implies that the composition
$\psi\circ\vf$ has finite $\text{G}_C$-dimension.  

\eqref{ciid04item2}
Apply part~\eqref{ciid04item1} in the case $C=R$.
\end{proof}

\section*{Acknowledgments}

The author is grateful to Lars W.\ Christensen,
Anders J.\ Frankild and Diana White for thoughtful discussions 
about this research and to 
Parviz Sahandi, Tirdad Sharif and Siamak Yassemi for sharing
a preliminary version of~\cite{sahandi:hfd}.
The author also thanks the anonymous referee for thoughtful comments.


\providecommand{\bysame}{\leavevmode\hbox to3em{\hrulefill}\thinspace}
\providecommand{\MR}{\relax\ifhmode\unskip\space\fi MR }
\providecommand{\MRhref}[2]{%
  \href{http://www.ams.org/mathscinet-getitem?mr=#1}{#2}
}
\providecommand{\href}[2]{#2}

\end{document}